\documentclass[oneside,english,a4wide]{amsart}
\usepackage[latin9]{inputenc}
\usepackage[a4paper]{geometry}
\usepackage{mathabx}
\usepackage{esvect}
\usepackage{babel}
\usepackage{mathtools}
\usepackage{amstext}
\usepackage{amsthm}
\usepackage{amssymb}
\usepackage{enumerate}
\usepackage{esint}
\usepackage{graphicx}
\usepackage{bbm}
\usepackage[all]{xy}
\usepackage{mathrsfs}
\usepackage[unicode=true,
bookmarks=true,bookmarksnumbered=true,bookmarksopen=true,bookmarksopenlevel=2,
breaklinks=false,pdfborder={0 0 1},backref=false,colorlinks=false]
{hyperref}
\hypersetup{
	colorlinks,linkcolor=blue,anchorcolor=blue,citecolor=blue}
\usepackage{breakurl}
\usepackage{autobreak}

\usepackage[svgnames]{xcolor} 

\usepackage{graphicx} 

\makeatletter
\numberwithin{equation}{section}
\numberwithin{figure}{section}
\theoremstyle{plain}
\newtheorem{thm}{\protect\theoremname}[section]
\theoremstyle{plain}

\newtheorem{proposition}[thm]{Proposition}
\newtheorem{corollary}[thm]{Corollary}
\theoremstyle{definition}

\theoremstyle{plain}
\newtheorem{lem}[thm]{\protect\lemmaname}
\newtheorem{lemma}[thm]{Lemma}
\theoremstyle{remark}
\newtheorem{rem}[thm]{\protect\remarkname}

\theoremstyle{definition}
\newtheorem{definition}[thm]{Definition}
\newtheorem*{claim*}{Claim}

\theoremstyle{remark}

\theoremstyle{definition}
\newtheorem*{defn*}{\protect\definitionname}

\usepackage{babel}
\usepackage{amscd}

\makeatother

\providecommand{\definitionname}{Definition}
\providecommand{\lemmaname}{Lemma}
\providecommand{\propositionname}{Proposition}
\providecommand{\remarkname}{Remark}
\providecommand{\theoremname}{Theorem}

\newcommand{\Rmnum}[1]{\expandafter\@slowromancap\romannumeral #1@}
\makeatother

\newcommand{\A}{{\mathcal A}}

\newcommand{\M}{{\mathcal M}}
\newcommand{\N}{{\mathcal N}}

\newcommand{\8}{\infty}

\newcommand{\be}{\begin{eqnarray*}}
	\newcommand{\ee}{\end{eqnarray*}}
\newcommand{\beq}{\begin{equation}}
	\newcommand{\eeq}{\end{equation}}
\newcommand{\beqn}{\begin{equation*}}
	\newcommand{\eeqn}{\end{equation*}}
\newcommand{\bs}{\begin{split}}
	\newcommand{\es}{\end{split}}

\begin{document}

	
	
	\title[Noncommutative
	Littlewood-Paley-Stein inequalities]{On the best constants of the noncommutative Littlewood-Paley-Stein inequalities}
	
	
	\author[Zhenguo Wei]{Zhenguo Wei}
	\address{
		Laboratoire de Math{\'e}matiques, Universit{\'e} de Bourgogne Franche-Comt{\'e}, 25030 Besan\c{c}on Cedex, France}
\email{zhenguo.wei@univ-fcomte.fr}

	\author[Hao Zhang]{Hao Zhang}
	\address{
	Department of Mathematics, University of Illinois Urbana-Champaign, USA}
	\email{hzhang06@illinois.edu}
	\date{}
	\maketitle
	
	\begin{abstract}
		Let $1<p<\8$. Let $\{T_t\}_{t>0}$ be a noncommutative symmetric diffusion semigroup on a semifinite von Neumann algebra $\mathcal{M}$, and let $\{P_t\}_{t>0}$ be its associated subordinated Poisson semigroup. The celebrated noncommutative Littlewood-Paley-Stein inequality asserts that for any $x\in L_p(\M)$,
		\begin{equation*}
		\alpha_p^{-1}\|x\|_{p}\le \|x\|_{p,P}\le \beta_p \|x\|_{p},
		\end{equation*}
		where $\|\cdot\|_{p,P}$ is the $L_p(\M)$-norm of square functions associated with  $\{P_t\}_{t>0}$, and $\alpha_p, \beta_p$ are the best constants only depending on $p$.
		
		We show that as $p\to \8$,
		$$ \beta_p\lesssim p,  $$
		and $p$ is the optimal possible order of $\beta_p$ as well. We also obtain some lower and upper bounds of $\alpha_p$ and $\beta_p$ in the other cases.
	\end{abstract}
	
	
\section{Introduction}\label{Introduction}

 This article follows the recent investigation of the vector-valued Littlewood-Paley-Stein theory developed by Xu. In his remarkable paper \cite{X}, he generalized the vector-valued Littlewood-Paley-Stein inequalities of symmetric Markovian semigroups to a more general class of semigroups of regular contractions. The most interesting aspect is that his new approach often yields the optimal orders of the relevant constants in most cases.

At first, we recall the classical Littlewood-Paley-Stein theory discovered by Stein in \cite{S}. Let $(\Omega,\mathcal{B},\mu)$ be a $\sigma$-finite measure space. $\{T_t\}_{t>0}$ is a symmetric diffusion semigroup if $\{T_t\}_{t>0}$ satisfies the following conditions:
\begin{enumerate}
	\item $T_t$ is a contraction on $L_p(\Omega)$ for every $1\le p\le \infty$,
	\item $T_tT_s=T_{t+s}$,
	\item $\lim\limits_{t\to0}T_t(f)=f$ in $L_2(\Omega)$ for every $f\in L_2(\Omega)$,
	\item $T_t$ is positive (i.e. positivity preserving),
	\item $T_t$ is selfadjoint on $L_2(\Omega)$,
	\item $T_t(1)=1$.
\end{enumerate}
Let $\{P_t\}_{t>0}$ be the Poisson semigroup subordinated to $\{T_t\}_{t>0}$ defined as follows
\begin{equation}\label{Ptdef}
	P_t(f)=\frac{1}{\sqrt{\pi}}\int_0^\infty \frac{e^{-s}}{\sqrt{s}}T_{\frac{t^2}{4s}}(f) ds, \quad \forall \ f\in L_p(\Omega),\,\,1\leq p\leq\infty.
\end{equation}
Then $\{P_t\}_{t>0}$ is also a symmetric diffusion semigroup. The square function associated with $\{P_t\}_{t>0}$, denoted by $G^{P}(f)$, is defined as
\begin{equation*}
G^{P}(f)=\biggl(\int_0^\infty \biggl|t\frac{\partial}{\partial t}P_t(f)\biggr|^2\frac{dt}{t}\biggr)^{1/2}, \quad \forall \ f\in L_p(\Omega).
\end{equation*}
Let $1<p<\infty$. The famous Littlewood-Paley-Stein inequality states that 
	\begin{equation}\label{Ptclass}
		a_p^{-1}\|f-\mathrm{F}(f)\|_{L_p(\Omega)}\le \|G^{P}(f)\|_{L_p(\Omega)}\le b_p\|f-\mathrm{F}(f)\|_{L_p(\Omega)}, \quad \forall \ f\in L_p(\Omega).
	\end{equation}
	where $\mathrm{F}$ is the projection from $L_p(\Omega)$ onto the fixed point subspace of $\{P_t\}_{t>0}$ of $L_p(\Omega)$, and $a_p, b_p$ are the corresponding best constants only depending on $p$.

Stein showed \eqref{Ptclass} by virtue of Burkholder-Gundy's inequality and Rota's dilation, which reveals a profound relationship between semigroup theory and martingale theory. However, Stein's proof does not yield the optimal orders of the best constants $a_p,b_p$. We refer the interested reader to \cite[Chapter 4, 5]{S} for more details.

Recently in \cite{X}, Xu proved the following estimates of $a_p$ and $b_p$:
\begin{equation}\label{apbp}
	a_p\lesssim \max\{p,{p'}^{\frac{1}{2}}\} \quad \text{and} \quad b_p\lesssim \max\{p^{\frac{1}{2}},p'\},
\end{equation}
where $p'=\frac{p}{p-1}$ is the conjugate exponent of $p$. Xu's new approach is built on holomorphic functional calculus and transference. More specifically, he first transferred a general semigroup to the special one of the translation group via Fendler's dilation. Then he exploited Calder\'{o}n-Zygmund theory, Wilson's intrinsic square functions and the weighted inequality theory to derive the estimates of the square functions associated with the translation group.  We would like to remark that the vector-valued Littlewood-Paley-Stein theory should be employed to obtain (\ref{apbp}), even though $a_p$ and $b_p$ are involved in the scalar setting.

Moreover, Xu also verified in \cite{X} that the above estimates of $b_p$ is optimal as $p\to 1$ and $p\to \infty$ when $\{P_t\}_{t>0}$ is the classical Poisson semigroup on $\mathbb{R}$. However, the optimal order of $a_p$ seems much more difficult to be determined. For example, when $\{P_t\}_{t>0}$ is the classical Poisson semigroup on $\mathbb{R}$, Xu showed in \cite{Xu2022} that 
$$ \sqrt{p}\lesssim a_p\lesssim p, \quad \forall \ 1\leq p<\8.  $$

In \cite{XZ}, the second named author with the coauthor demonstrated that $a_p\approx p$ is optimal as $p\to \infty$ for symmetric diffusion semigroups. Their argument is based on the construction of a special symmetric diffusion semigroup $\{T_t\}_{t>0}$ associated with any given martingale such that its square function $G^T(f)$ for any $f\in L_p(\Omega)$ is pointwise comparable with the martingale square function of $f$. Hence they derived the lower bound of $a_p$ from the best constants involved in the Burkholder-Gundy inequality for martingales. 

Then it still remains an open question to determine the optimal order of $a_p$ as $p\to \8$ when $\{P_t\}_{t>0}$ is the classical Poisson semigroup, and to determine the optimal order of $a_p$ as $p\to 1$ for any symmetric diffusion semigroup.

\

Motivated by the aforementioned work, we are concerned with the optimal orders of the best constants for the noncommutative Littlewood-Paley-Stein inequalities. First of all, we recall the noncommutative symmetric diffusion semigroups and the corresponding noncommutative Littlewood-Paley-Stein inequalities.

Let $\mathcal{M}$ be a von Neumann algebra equipped with a normal semifinite faithful trace $\tau$. $\{T_t\}_{t>0}$ is a noncommutative symmetric diffusion semigroup if it satisfies the following conditions:
\begin{enumerate}[(a)]
	\item each $T_t:\mathcal{M}\to\mathcal{M}$ is unital, normal and completely positive,
	\item for any $x\in\mathcal{M}$, $T_t(x)\to x$ in the $w^*$-topology of $\mathcal{M}$ when $t\to 0^+$,
	\item each $T_t:\mathcal{M}\to\mathcal{M}$ is selfadjoint. Namely, for any $x,y\in \mathcal{S}(\mathcal{M})$
	\begin{equation*}
		\tau(T_t(x)y)=\tau(xT_t(y)),
	\end{equation*} 
    \item the extension of each $T_t:\mathcal{M}\to\mathcal{M}$ from $L_p(\mathcal{M})$ to $L_p(\mathcal{M})$ for $1\le p<\infty$ is completely contractive.
\end{enumerate}
It is well-known that $\{T_t\}_{t>0}$ extends to a semigroup of completely positive contractions on $L_p(\mathcal{M})$ for any $1\le p<\infty$ by \cite[Lemma 1.1]{JX2}. Moreover, $\{T_t\}_{t>0}$ is strongly continuous on $L_p(\mathcal{M})$ for any $1\le p<\infty$ by \cite[Proposition 1.23]{D} or \cite{JX2}.

Let $1<p<\8$. For any given $x\in L_p(\mathcal{M})$, we define the column and row square functions associated with $\{T_t\}_{t>0}$ respectively by
	\begin{equation*}
		G_c^{T}(x)=\biggl(\int_0^\infty \big|t\frac{\partial}{\partial t}T_t(x)\big|^2\frac{dt}{t}\biggr)^{1/2}
	\end{equation*}
	and
    \begin{equation*}
    	G_r^{T}(x)=\biggl(\int_0^\infty \bigl|t\frac{\partial}{\partial t} T_t(x)^*\bigr|^2\frac{dt}{t}\biggr)^{1/2}.
    \end{equation*}
Furthermore, the $p$-norm of $x\in L_p(\mathcal{M})$ for square functions of noncommutative semigroups, denoted by $\|\cdot\|_{p,T}$, is defined as follows: 
\begin{enumerate}
	\item if $2\le p<\infty$,
	\begin{equation*}
		\|x\|_{p,T}=\max \Big\{\|G_c^T(x)\|_{p},\|G_r^T(x)\|_{p}\Big\},
	\end{equation*}
    \item if $1< p< 2$,
    \begin{equation*}
    	\begin{aligned}
    		\|x\|_{p,T}{}&=\inf \Big\{ \|G_c^T(y)\|_{p}+\|G_r^T(z)\|_{p}:x=y+z \Big\}.
    	\end{aligned}
    \end{equation*}
\end{enumerate}

In \cite{JX}, Junge, Le Merdy and Xu showed the following noncommutative Littlewood-Paley-Stein inequality: if $1<p<\infty$ and $\{T_t\}_{t>0}$ is a noncommutative symmetric diffusion semigroup, then
\begin{equation}\label{nlps}
	\alpha_p^{-1}\|x\|_{p}\le \|x\|_{p,T}\le \beta_p \|x\|_{p},
\end{equation}
where $\alpha_p,\beta_p$ are the corresponding best constants only depending on $p$. (We assume that $A$ has dense range where $T_t=\exp(-t A)$ $({t>0})$ on $L_p(\M)$, see Remark \ref{rem217}.) Let $\{P_t\}_{t>0}$ be the subordinated Poisson semigroup associated with $\{T_t\}_{t>0}$. Then $\{P_t\}_{t>0}$ is also a noncommutative symmetric diffusion semigroup on $\M$. Hence, (\ref{nlps}) holds for $\{P_t\}_{t>0}$ as well.

In this article, we aim to obtain the optimal orders of magnitude on $p$ of best constants $\alpha_p,\beta_p$ in (\ref{nlps}) in terms of the subordinated Poisson semigroup $\{P_t\}_{t>0}$. The following theorem is our main result.
\begin{thm}\label{main1}
	Let $1< p<\infty$. Suppose that $\{T_t\}_{t>0}$ is a noncommutative symmetric diffusion semigroup on $\mathcal{M}$ and $\{P_t\}_{t>0}$  its subordinated Poisson semigroup. Let $\alpha_p,\beta_p$ be the best positive constants such that
	\begin{equation*}
		\alpha_p^{-1}\|x\|_{p}\le\|x\|_{p,P}\le \beta_p \|x\|_{p},\quad \forall \ x\in L_p(\mathcal{M}).
	\end{equation*}
Then we have 
$$ p\lesssim \alpha_p\lesssim \max\{p^2,p'\} \quad \text{and } \quad \max\{p,p'\}\lesssim\beta_p\lesssim \max\{p,p'^3\}. $$
Moreover, $p$ is the optimal possible order of $\beta_p$ as $p\to\infty$. 
\end{thm}	


It has been shown that $\beta_p\lesssim (p')^6$ as $p\to 1$ in \cite{GSX}. So we improve their result to $(p')^3$ as $p\to 1$. When $\M=L_\8(\Omega)$, a noncommutative symmetric diffusion semigroup is a classical symmetric diffusion semigroup. Then $\alpha_p\geq a_p$ and $\beta_p\geq b_p$. In addition, it is easy to deduce from (\ref{apbp}) and \cite[Corollary 1.6]{XZ} that
$$ \alpha_p\gtrsim p \quad \text{and} \quad \beta_p\gtrsim\max\{p^{\frac{1}{2}},p'\}.  $$
However, from Theorem \ref{main1}, we see $\beta_p\approx p$ as $p\to \8$, which reflects the difference between the best constants in the noncommutative setting and those in the classical setting. A detailed discussion about the lower bounds of $\alpha_p$ and $\beta_p$ is included in Corollary \ref{lower}.

When $1<p<2$, there is an alternative $p$-norm denoted by $[\,\cdot\,]_{p,T}$  of $x\in L_p(\mathcal{M})$ for semigroup square functions by letting
\begin{equation*}
	\begin{aligned}
		[x]_{p,T}=\inf \Big\{ \|u_1\|_{L_p(\mathcal{M};(L_2(\mathbb{R}_+,\frac{dt}{t}))^{c})}&+\|u_2\|_{L_p(\mathcal{M};(L_2(\mathbb{R}_+,\frac{dt}{t}))^r)}:\\
		&u_1(t)+u_2(t)=-t\frac{\partial}{\partial t}T_t(x),\,\,\forall \ t\in\mathbb{R}_+ \Big\}.
	\end{aligned}
\end{equation*}
It is obvious that $[x]_{p,T}\leq  \|x\|_{p, T}$. In particular, if $\M=L_\8(\Omega)$, then $[x]_{p,T}=\|x\|_{p, T}$. 

From \cite[Theorem 7.8]{JX}, if $\{T_t\}_{t>0}$ is a noncommutative symmetric diffusion semigroup on $\mathcal{M}$, then for $1<p<2$ and any $x\in L_p(\M)$,
$$ \|x\|_{p, T}\approx_p [x]_{p,T}.  $$
Let $1<p<2$. Let $\widetilde{\alpha}_p, \widetilde{\beta}_p$ be the best constants such that
$$ \widetilde{\alpha}_p^{-1}\|x\|_{p}\le[x]_{p,P}\le \widetilde{\beta}_p \|x\|_{p},\quad \forall \ x\in L_p(\mathcal{M}).$$
Then $\widetilde{\alpha}_p\geq \alpha_p$. From \cite[Section 10.C-10.D]{JX} or \cite[Proposition 2.10]{GSX}, one has
\begin{equation}\label{ppppp}
 \widetilde{\beta}_p\lesssim p'^2. 
 \end{equation}
 
 \begin{corollary}\label{widealpha}
 	When $1<p<2$, we have
 	$$\widetilde{\alpha}_p\lesssim p'.$$
 	\end{corollary}

 The estimate of $\widetilde{\alpha}_p$ is implied by Theorem \ref{main1} and duality, which we will demonstrate in Section \ref{est2}. Corollary \ref{widealpha} also yields $\alpha_p\lesssim p'$ as $p\to 1$.

The paper is organized as follows. Section \ref{Preliminaries} provides an overview of notation and background, including noncommutative $L_p$-spaces, noncommutative symmetric diffusion semigroups and dilatable semigroups, noncommutative Burkholder-Gundy inequality and holomorphic functional calculus. Section \ref{est1} is devoted to estimating $\alpha_p$ and $\beta_p$ for $2\le p<\infty$. We mainly adopt Xu's approach in \cite{X}. To be more specific, by utilizing the dilatable property of $\{T_t\}_{t>0}$, we dilate our semigroup $\{T_t\}_{t>0}$ into a group of isometries. We then use the transference method to reduce our problem to the case of the translation group, and ultimately apply the semicommutative singular integral operator theory known in \cite{MP} to get the desired result. In section \ref{est2}, we employ the methods outlined in \cite{JX} and \cite{X} to estimate $\alpha_p,\widetilde{\alpha}_p$ and $\beta_p$. The key of these ingredients is to prove that $\sqrt{A}$, where $A$ is the negative infinitesimal generator of $\{T_t\}_{t>0}$, is Col-sectorial of Col-type $\frac{\pi}{4}$ and Row-sectorial of Row-type $\frac{\pi}{4}$ (see Proposition \ref{mainlemma}).

Throughout this paper, we shall use the following notation: $A\lesssim B$ (resp. $A\lesssim_\varepsilon B$) means that $A\le CB$ (resp. $A\le C_\varepsilon B$) for some absolute positive constant $C$ (resp. a positive constant $C_\varepsilon$ depending only on $\varepsilon$). $A\approx B$ or $A\approx_\varepsilon B$ means that these inequalities as well as their inverses hold.	Denote by $p'$ the conjugate exponent of $p$. For simplicity, we use the abbreviation $\partial =\frac{\partial}{\partial t}$.

\bigskip

\section{Preliminaries}\label{Preliminaries}
\subsection{\textbf{Noncommutative $L_p$-spaces.}}
Let $\mathcal{M}$ be a von Neumann algebra equipped with a normal semifinite faithful trace $\tau$. Denote by $\M_+$ the positive part of $\M$. Let $\mathcal{S}_+(\M)$ be the set of all $x\in \M_+$ whose support projection has a finite trace, and $\mathcal{S}(\mathcal{M})$ be the linear span of $\mathcal{S}_+(\mathcal{M})$. Then $\mathcal{S}(\mathcal{M})$ is a $w^*$-dense $*$-subalgebra of $\mathcal{M}$. Let $x\in\mathcal{S}(\mathcal{M})$, then $|x|^p\in\mathcal{S}(\mathcal{M})$ for any $0<p<\infty$, where $|x|:=(x^*x)^{1/2}$. Define
\[\|x\|_p=(\tau(|x|^p))^{1/p}.\]
Thus $\|\cdot\|_p$ is a norm for $p\ge 1$, and a $p$-norm for $0<p<1$. The noncommutative $L_p$-space denoted by $L_p(\mathcal{M},\tau)$ associated with $(\mathcal{M},\tau)$ is the completion of $(\mathcal{S}(\M),\|\cdot\|_p)$ for $0<p<\8$. We also write $L_p(\mathcal{M},\tau)$ simply by $L_p(\mathcal{M})$ for short. When $p=\8$, we set $L_\infty(\mathcal{M}):=\mathcal{M}$ equipped with the operator norm. In particular, when $p=2$, $L_2(\M)$ is a Hilbert space. We refer the reader to \cite{PX} for a detailed exposition of noncommutative $L_p$-spaces.

Now we present the tensor product of von Neumann algebras. Assume that each $\mathcal{M}_k$ $(k=1, 2)$ is equipped with a normal semifinite faithful trace $\tau_k$. Let $H_k=L_2\left(\mathcal{M}_k\right)$. We consider $\mathcal{M}_k$ as a von Neumann algebra acting on $H_k$ by left multiplication, which means $\M_k \hookrightarrow B(H_k)$ via the embedding $x\longmapsto L_x\in B(H_k)$, where $x\in \M_k$ and $L_x(y):= x\cdot y\in H_k$ for any $y\in H_k$. Then the tensor product of $\M_1$ and $\M_2$ denoted by $\mathcal{M}_1 \otimes \mathcal{M}_2$ is the $w^*$-closure of $\text{span}\{x_1\otimes x_2 | x_1\in \M_1, x_2\in \M_2\}$ in $B(L_2(\M_1)\otimes L_2(\M_2))$. Here $L_2(\M_1)\otimes L_2(\M_2)$ is the Hilbert space tensor product of $L_2(\M_1)$ and $ L_2(\M_2)$.

It is well-known that there exists a unique normal semifinite faithful trace $\tau$ on the von Neumann algebra tensor product $\mathcal{M}_1 \otimes \mathcal{M}_2$ such that
$$
\tau\left(x_1 \otimes x_2\right)=\tau_1\left(x_1\right) \tau_2\left(x_2\right), \quad \forall \ x_1 \in \mathcal{S}(\M_1),\,\, \forall \ x_2 \in \mathcal{S}(\M_2) .
$$
$\tau$ is called the tensor product of $\tau_1$ and $\tau_2$ and denoted by $\tau_1 \otimes \tau_2$.

Let $H$ be a Hilbert space. Now we define various $H$-valued noncommutative $L_p$-spaces. Let the von Neumann algebra $B(H)$ be equipped with its usual trace $\mathrm{Tr}$. We equip the von Neumann algebra $\mathcal{M}\otimes B(H)$ with the trace $\tau\otimes \mathrm{Tr}$. For any $a,b\in H$, define the rank one operator as follows
\[a\otimes  b(\xi)=\langle b,\xi\rangle a, \quad \forall \ \xi\in H.\]
We fix some $e\in H$ with $\|e\|=1$, and let $p_e=e\otimes e\in B(H)$ be the rank one orthogonal projection onto $\mathrm{span}\{e\}$. Then for any $0<p\le \infty$, define the following column spaces and row spaces respectively,
\begin{equation*}
	L_p(\mathcal{M};H^c)=L_p(\mathcal{M}\otimes B(H))(1\otimes p_e),
\end{equation*}
and 
\begin{equation*}
	L_p(\mathcal{M};H^r)=(1\otimes p_{e})L_p(\mathcal{M}\otimes B(H)).
\end{equation*}
This definition is essentially independent of the choice of $e$. We regard $L_p(\mathcal{M})$ as a subspace of $L_p(\mathcal{M}\otimes B(H))$ by identifying any $c\in L_p(\mathcal{M})$ with $c\otimes p_e$. This is equivalent to writing that
\begin{equation*}
	L_p(\mathcal{M})=(1\otimes p_e)L_p(\mathcal{M}\otimes B(H))(1\otimes p_e).
\end{equation*}
Then for any element $u\in L_p(\mathcal{M};H^c)$, $$u^*u\in (1\otimes p_e)L_{p/2}(\mathcal{M}\otimes B(H))(1\otimes p_e)\subset L_{p/2}(\mathcal{M}\otimes B(H)). $$
Due to the above identification for $p/2$, we may therefore see $u^*u$ as an element of $L_{p/2}(\mathcal{M})$, which implies that
\begin{equation}\label{ustar}
	\|u\|_{L_p(\mathcal{M};H^c)}=\|(u^*u)^{1/2}\|_{p}.
\end{equation}
Similarly, if $u\in L_p(\mathcal{M};H^r)$, then one has
\begin{equation}
\|u\|_{L_p(\mathcal{M};H^r)}=\|(uu^*)^{1/2}\|_{p}.
\end{equation}
Arguing as in \cite[Chapter 2]{JX}, we may use these identities to regard the algebraic tensor product $L_p(\mathcal{M}) \otimes_{alg} H$ as a dense subspace of $L_p(\mathcal{M};H^c)$ and $L_p(\mathcal{M};H^r)$. In particular, when $p=2$, it is clear that $L_2(\mathcal{M};H^c)=L_2(\mathcal{M};H^r)$. 

Let $1\le p,q\le \infty$ and $0<\theta<1$. If $\frac{1}{s}=\frac{1-\theta}{p}+\frac{\theta}{q}$, then we obtain
\begin{equation}\label{compinter}
	[L_p(\mathcal{M};H^c),L_q(\mathcal{M};H^c)]_{\theta}=L_s(\mathcal{M};H^c),
\end{equation}
where $[\,\,,\,\,]_\theta$ stands for the interpolation space obtained by the complex interpolation method (see \cite{BJLJ}). We would like to remark that column spaces $L_p\left(\mathcal{M} ; H^c\right)$ and row spaces $L_p\left(\mathcal{M} ; H^r\right)$ play symmetric roles in most cases, for instance, the interpolation property similar to (\ref{compinter}) also holds for row spaces $L_p\left(\mathcal{M} ; H^r\right)$. In the sequel, we will often state some results for $L_p\left(\mathcal{M} ; H^c\right)$ only and then take granted that they have also a row version, unless specified otherwise.

For $1\le p<\infty$, it is well-known that 
\begin{equation}\label{duality1}
	L_p(\mathcal{M};H^c)^*=L_{p'}(\mathcal{M};H^r),
\end{equation}
where the duality pairing is defined by taking $(y\otimes b,x\otimes a)$ to $\langle \overline{b},a\rangle \tau(yx)$ for any $x\in L_p(\mathcal{M})$, $y\in L_{p'}(\mathcal{M})$ and $a,b\in H$.

 Now we need to introduce two more $H$-valued noncommutative $L_p$-spaces, namely the intersection and the sum of row and column spaces. For $2\le p<\infty$, we define the intersection
\begin{equation*}
	L_p(\mathcal{M};H^{r\cap c})=L_p(\mathcal{M};H^c)\cap L_p(\mathcal{M};H^r),
\end{equation*}
equipped with the norm
\begin{equation*}
	\|u\|_{L_p(\mathcal{M};H^{r\cap c})}=\max \{\|u\|_{L_p(\mathcal{M};H^c)},\|u\|_{L_p(\mathcal{M};H^r)}\}.
\end{equation*}
Then for $1\leq p\leq 2$, we define the sum
\begin{equation*}
	L_p(\mathcal{M};H^{r+c})=L_p(\mathcal{M};H^c)+ L_p(\mathcal{M};H^r),
\end{equation*}
equipped with the norm
\begin{equation*}
	\|u\|_{L_p(\mathcal{M};H^{r+c})}=\inf \{\|u_1\|_{L_p(\mathcal{M};H^c)}+\|u_2\|_{L_p(\mathcal{M};H^r)}:u=u_1+u_2\}.
\end{equation*}
When $p=2$, it is easy to see that $	L_2(\mathcal{M};H^{r\cap c})=L_2(\mathcal{M};H^{r+c})$.

Let $1\le p<\infty$. In particular, if $H=\ell_2$ and $e=(\delta_{1i})_{1\leq i<\8}\in \ell_2$, then for any finite sequence $\{x_k\}_{k=1}^n\subset L_p(\mathcal{M})$, we have
\begin{equation*}
	\|\{x_k\}_{k=1}^n\|_{L_p(\mathcal{M};\ell_2^c)}=\Big\|\Big(\sum_{k=1}^n x_k^*x_k\Big)^{1/2}\Big\|_{p}
\end{equation*}
and
\begin{equation*}
	\|\{x_k\}_{k=1}^n\|_{L_p(\mathcal{M};\ell_2^r)}=\Big\|\Big(\sum_{k=1}^n x_kx_k^*\Big)^{1/2}\Big\|_{p}.
\end{equation*}
Then $L_p(\mathcal{M};\ell_2^c)$ (resp. $L_p(\mathcal{M};\ell_2^r)$) is the completion of the space of all finite sequences in $L_p(\mathcal{M})$ with respect to $\|\cdot\|_{L_p(\mathcal{M};\ell_2^c)}$ (resp. $\|\cdot\|_{L_p(\mathcal{M};\ell_2^c)}$). Moreover, one has
\begin{equation}\label{duality2}
	L_p(\mathcal{M};\ell_2^c)^*=L_{p'}(\mathcal{M};\ell_2^r) \quad \text{and} \quad L_p(\mathcal{M};\ell_2^r)^*=L_{p'}(\mathcal{M};\ell_2^c)
\end{equation}
isometrically with respect to the duality bracket:
\begin{equation*}
	\big\langle \{y_n\}_{n \in\mathbb{N}},\{x_n\}_{n\in\mathbb{N}} \big\rangle \mapsto \sum_{n\in\mathbb{N}} \tau(y_nx_n),\quad\quad \forall \ \{x_n\}_{n\in\mathbb{N}}\subset L_p(\mathcal{M}),\,\, \{y_n\}_{n\in\mathbb{N}}\subset L_{p'}(\mathcal{M}).
\end{equation*}
Besides, $L_p(\mathcal{M};\ell_2^{r\cap c})$ and $L_p(\mathcal{M};\ell_2^{r+c})$ are the intersection and the sum of $L_p(\mathcal{M};\ell_2^{c})$ and $L_p(\mathcal{M};\ell_2^{r})$ respectively.
We refer the interested reader to \cite{JX,P} for more information.

\

Let $(\Omega,\mu)$ be a $\sigma$-finite measure space, and let  $X$ be a Banach space. We now introduce vector-valued $L_p$-space $L_p(\Omega,X)$. For any $1\le p<\infty$, $L_p(\Omega,X)$ consists of all (strongly) measurable functions $u:\Omega\to X$ such that $\int_\Omega\|u(t)\|_{X}^p d\mu(t)$ is finite. The norm on this space is given by 
\[\|u\|_{L_p(\Omega,X)}=\biggl(\int_\Omega\|u(t)\|_{X}^p d\mu(t)\biggr)^{1/p},\quad \forall \ u\in L_p(\Omega,X).\]
For any $u\in L_p(\Omega,X)$ and $v\in L_{p'}(\Omega,X^*)$, the function $t\mapsto \langle v(t),u(t)\rangle$ is integrable, and we define a duality pairing
\begin{equation*}
	\langle v,u\rangle=\int_\Omega \langle v(t),u(t)\rangle d\mu(t),
\end{equation*}
which implies an isometric inclusion
\begin{equation*}
	L_{p'}(\Omega,X^*) \hookrightarrow L_p(\Omega,X)^*.
\end{equation*}
If further $X$ is reflexive, from \cite[Chapter IV]{DU} we obtain an isometric isomorphism $L_{p'}(\Omega,X^*)= L_p(\Omega,X)^*$.

When $X=L_p(\mathcal{M})$, it is easy to verify that $L_p(\Omega,L_p(\mathcal{M}))$ isometrically coincides with $L_p(L_\infty(\Omega)\otimes \mathcal{M})$.
We refer the reader to \cite{PX,JX} for more details.




\subsection{Dilatable semigroups}\label{Dilatable semigroup}
This subsection is devoted to showing that each noncommutative symmetric diffusion semigroup is dilatable with the help of the factorizable property for Markov maps. We begin with the definition of dilatable semigroups and Markov maps.
\begin{definition}
	Let $1<p<\infty$, and let $\mathcal{M}$ be a von Neumann algebra equipped with a normal semifinite faithful trace $\tau$. Suppose that $\{T_t\}_{t>0}$ is a positive contractive and strong continuous semigroup on $L_p(\mathcal{M})$. We say $\{T_t\}_{t>0}:L_p(\mathcal{M})\to L_p(\mathcal{M})$ is a dilatable semigroup, if there exists another semifinite von Neumann algebra $\mathcal{N}$, a strong continuous group $\{U_t\}_{t>0}$ of completely positive isometries on $L_p(\mathcal{N})$, and completely positive contractive maps $J:L_p(\mathcal{M})\to L_p(\mathcal{N})$ and $Q:L_p(\mathcal{N})\to L_p(\mathcal{M})$ such that
	\begin{equation*}
	T_t=QU_tJ,\quad \forall \ t>0.
	\end{equation*}
\end{definition}

The following definitions of these operators are considered in \cite{ADC,Haar}.
\begin{definition}
	Let $(\M, \tau)$ and $(\N, \phi)$ be von Neumann algebras equipped with normal faithful normalized traces $\tau$ and $\phi$, respectively. A linear map $T: \M \rightarrow \N$ is called a $(\tau, \phi)$-Markov map if
	\begin{enumerate}
		\item $T$ is completely positive,
		\item $T$ is unital,
		\item $\phi \circ T=\tau$.
	\end{enumerate}
	In particular, when $(\M, \tau)=(\N, \phi)$, we say that $T$ is a $\tau$-Markov map.
\end{definition}

\begin{definition}
	A $(\tau, \phi)$-Markov map $T: \M \rightarrow \N$ is called factorizable if there exists a von Neumann algebra $P$ equipped with a normal faithful normalized trace $\chi$, and $*$-monomorphisms $J_0: M \rightarrow P$ and $J_1: N \rightarrow P$ such that $J_0$ is $(\tau, \chi)$-Markov and $J_1$ is $(\phi, \chi)$-Markov, satisfying, moreover, $T=J_1^* \circ J_0$.
\end{definition}

The following definition is due to K\"{u}mmerer in \cite{KU}.
\begin{definition}
	Let $\M$ be a von Neumann algebra with a normal faithful normalized trace $\tau$ and let $T: \M \rightarrow \M$ be a $\tau$-Markov map. A dilation of $T$ is a quadruple $(\N, \phi, \alpha, \iota)$, where $\N$ is a von Neumann algebra with a normal faithful normalized trace $\phi$, $\alpha$ is an automorphism of $\N$ leaving $\phi$ invariant and $\iota: \M \rightarrow \N$ is a $(\tau, \phi)$-Markov *-monomorphism, satisfying
	$$
	T^n=\iota^* \circ \alpha^n \circ \iota, \quad \forall 
	 \ n \geq 0.
	$$
\end{definition}

The following theorem established by Haagerup and Musat in \cite[Theorem 4.4]{Haar} describes the equivalence of the factorizable property and the dilation property for $\tau$-Markov maps.
\begin{thm}\label{HMu}
	Let $\M$ be a von Neumann algebra with a normal faithful normalized trace $\tau$ and let $T: \M \rightarrow \M$ be a $\tau$-Markov map. The following statements are equivalent:
	\begin{enumerate}
		\item $T$ is factorizable,
		\item $T$ has a dilation.
	\end{enumerate}
\end{thm}

\begin{lem}\label{FAC}
	Let $\M$ be a von Neumann algebra with a normal faithful normalized trace $\tau$. Let $\{T_t\}_{t>0}$ be a noncommutative symmetric diffusion semigroup on $\mathcal{M}$. Then for each $t>0$, ${T_t}$ is factorizable.
\end{lem}
\begin{proof}
	Note that $\{{T_t}\}_{t>0}$ is a semigroup of selfadjoint $\tau$-Markov maps. This is an immediate consequence of \cite[Theorem 2.26]{JRS}.
\end{proof}

\begin{definition}
	Suppose $1 \leqslant p<\infty$. We say that a completely positive contraction $T: L^p(\M) \rightarrow L^p(\M)$ on a noncommutative $L^p$-space $L^p(\M)$ is completely positively dilatable if there exist a noncommutative $L^p$-space $L^p(\N)$, two completely positive contractions $J: L^p(\M) \rightarrow L^p(\N)$ and $P: L^p(\N) \rightarrow L^p(\M)$ and a completely positive invertible isometry $U: L^p(\N) \rightarrow L^p(\N)$ with $U^{-1}$ completely positive such that
	$$
	T^n=P U^n J, \quad \forall \ n \geqslant 0 .
	$$
\end{definition}

\begin{corollary}\label{coro288}
	Let $\M$ be a von Neumann algebra with a normal faithful normalized trace $\tau$. Let $\{T_t\}_{t>0}$ be a noncommutative symmetric diffusion semigroup on $\mathcal{M}$. Then for any $t>0$, ${T_t}$ is completely positively dilatable.
\end{corollary}
\begin{proof}
	By Lemma \ref{FAC} and Theorem \ref{HMu}, ${T_t}$ has a dilation, and then is completely positively dilatable.
\end{proof}

Arhancet gave the following theorem in \cite[Theorem 5.4]{ADC}, which allows us to transit from the dilation properties of individual operators to the dilatable properties of a semigroup.
\begin{thm}\label{dilcom}
	Suppose $1<p<\infty$. Let $\{T_t\}_{t > 0}$ be a strongly continuous semigroup of completely positive contractions on a noncommutative $L^p$-space $L^p(\M)$ such that each $T_t: L^p(\M) \rightarrow L^p(\M)$ is completely positively dilatable. Then there exists a noncommutative $L^p$-space $L^p(\N)$, a strongly continuous group of completely positive isometries $U_t: L^p(\N) \rightarrow L^p(\N)$ and two completely positive contractions $J: L^p(\M) \rightarrow L^p(\N)$ and $P: L^p(\N) \rightarrow L^p(\M)$ such that
	$$
	T_t=P U_t J, \quad \forall \  t > 0
	$$
\end{thm}

\begin{thm}\label{thm210}
	Let $\M$ be a von Neumann algebra with a normal faithful normalized trace $\tau$. If $\{T_t\}_{t>0}$ be a noncommutative symmetric diffusion semigroup on $\mathcal{M}$, then $\{{T_t}\}_{t>0}$ is dilatable on $L_p(\M)$ for any $1<p<\8$.
	\end{thm}
\begin{proof}
	It follows from Corollary \ref{coro288} and Theorem \ref{dilcom}.
	\end{proof}

\subsection{Noncommutative Burkholder-Gundy inequality}\label{NBGI} In this subsection, we will apply the best constants of the Burkholder-Gundy inequality to obtain the lower bounds of the best constants of the noncommutative Littlewood-Paley-Stein inequality.

First we recall noncommutative martingales. Let $\mathcal{M}$ denote a finite von Neumann algebra equipped with a normal faithful normalized trace $\tau$, and $\{\mathcal{M}_{n}\}_{n \geq 1}$ an increasing filtration of von Neumann subalgebras of $\mathcal{M}$ whose union is $w^{*}$-dense in $\mathcal{M}$. As usual, $L_{p}\left(\mathcal{M}_{n}\right)=L_{p}\left(\mathcal{M}_{n},\left.\tau\right|_{\mathcal{M}_{n}}\right)$ is naturally identified as a subspace of $L_{p}(\mathcal{M})$. It is well-known that there exists a unique normal faithful conditional expectation $\mathcal{E}_{n}$ from $\mathcal{M}$ onto $\mathcal{M}_{n}$ such that $\tau \circ \mathcal{E}_{n}=\tau .$ Moreover, $\mathcal{E}_{n}$ extends to a contractive projection from $L_{p}(\mathcal{M})$ onto $L_{p}\left(\mathcal{M}_{n}\right)$, for every $1 \leq p<\infty$, which is still denoted by $\mathcal{E}_{n}$. Similarly, we denote by $d\mathcal{E}_n=\mathcal{E}_{n}-\mathcal{E}_{n-1}$ for $n\geq 1$ the martingale difference with $\mathcal{E}_0=0$ as convention.

A noncommutative martingale with respect to $\{\mathcal{M}_{n}\}_{n\geq 1}$ is a sequence $x=\{x_{n}\}_{n \geq 1}$ in $L_{1}(\mathcal{M})$ such that
$$
x_{n}=\mathcal{E}_{n}\left(x_{n+1}\right), \quad \forall \ n \geq 1.
$$
The difference sequence of $x$ is $d x=\{d x_{n}\}_{n \geq 1}$, where $d x_{n}=x_{n}-x_{n-1}$ (with $x_{0}=0$ by convention). If in addition $x_n\in L_p(\M)$ with $1\leq p<\8$, $x$ is called an $L_{p}$-martingales with respect to $\{\M_n\}_{n\geq1}$. In this case, we set for $1\leq p<\8$
$$ \|x\|_p=\sup_{n\geq 1}\|x_n\|_p.   $$

In the sequel, we will fix $\left(\mathcal{M}, \tau\right)$ and $\{\mathcal{M}_{n}\}_{n \geq 1}$ as above and all noncommutative martingales will be with respect to $\{\mathcal{M}_{n}\}_{n \geq 1}$.


Now we introduce norms in Hardy spaces of martingales defined in \cite{PX2}. Let $1\le p<\infty$ and $x=\{x_n\}_{n\ge 1}$ be an $L_p$-martingale. Set,
for $p\ge 2$,

\begin{align*}
\|x\|_{\mathcal{H}_p(\mathcal{M})}&=\|\{dx_n\}\|_{L_p(\mathcal{M};\ell_2^{r\cap c})}\\
&=\max\bigg\{\bigg\|\biggl(\sum_{n= 1}^\8|dx_n|^2\biggr)^{1/2}\bigg\|_{p},\bigg\|\biggl(\sum_{n= 1}^\8|dx_n^*|^2\biggr)^{1/2}\bigg\|_{p}\bigg\},
\end{align*}
and for $1\leq p<2$,
\begin{align*}
\|x\|_{\mathcal{H}_p(\mathcal{M})}&=\|\{dx_n\}\|_{L_p(\mathcal{M};\ell_2^{r+c})}\\
&=\inf\bigg\{\bigg\|\biggl(\sum_{n= 1}^\8|dy_n|^2\biggr)^{1/2}\bigg\|_{p}+\bigg\|\biggl(\sum_{n= 1}^\8|dz_n^*|^2\biggr)^{1/2}\bigg\|_{p}\bigg\},
\end{align*}
where the infimum runs over all decompositions $x=y+z$ as sums of two $L_p$-martingales.

 In \cite{PX2}, Pisier and Xu proved the following celebrated noncommutative Burkholder-Gundy inequality.
\begin{thm}\label{BGI}
	Let $1<p<\infty$, and let $x=\{x_n\}_{n\ge 1}$ be an $L_p$-martingale. Then
	\begin{equation*}
	\theta_p^{-1}\|x\|_{\mathcal{H}_p(\mathcal{M})}\le \|x\|_p\le \gamma_p \|x\|_{\mathcal{H}_p(\mathcal{M})},
	\end{equation*}
	where $\theta_p$ and $\gamma_p$ depend only on $p$ and are the best constants to satisfy the above noncommutative Burkholder-Gundy inequality.
\end{thm}

Later on, in \cite{JX4}, Junge and Xu gave an almost complete picture on the optimal orders of $\theta_p$ and $\gamma_p$, except the optimal order of $\theta_p$ as $p\to 1$. Finally, Randrianantoanina in \cite{RN2} addressed the only remaining unsolved case on the optimal order of $\theta_p$ as $p\to 1$. To summarize the aforementioned results, we present the following theorem.
\begin{thm}\label{JXR}
	Let $1<p<\infty$. Let $\theta_p$ and $\gamma_p$ be as in Theorem \ref{BGI}. Then
	\begin{enumerate}
		\item $\theta_p\approx p$ as $p\to\infty$ and $\theta_p\approx p'$ as $p\to 1$,
		\item $\gamma_p\approx p$ as $p\to\infty$ and $\gamma_p\approx 1$ as $p\to 1$.
	\end{enumerate}
\end{thm}

Now we show the lower bounds of $\alpha_p$ and $\beta_p$. We follow the idea in \cite{XZ}. To this end, we first introduce a particular semigroup constructed by conditional expectations. Given a strictly increasing sequence $ \{ a_n \}_{n\geqslant 0} $ with $ a_0 = 0 $ and $ \lim\limits_{n \to \infty} a_n = 1  $, we define the mapping $\widetilde{T_t} \ (t>0)$ as follows
	\begin{equation}\label{eee3}
	\widetilde{T_t}=\sum\limits_{n= 1}^{\infty}(1-a_{n-1})^td \mathcal{E}_n=\sum\limits_{n= 1}^{\infty}\left[(1-a_{n-1})^t-(1-a_n)^t\right]\mathcal{E}_n. 
	\end{equation}
	It is clear that $\{\widetilde{T_t}\}_{t>0}$ is a noncommutative symmetric diffusion semigroup of $\tau$-Markov maps. Let $\{\widetilde{P}_t\}_{t>0}$ be the Poisson semigroup subordinated to $\{\widetilde{T}_t\}_{t>0}$. The following theorem established by Xu and Zhang in \cite{XZ} implies the equivalence of the corresponding norms between martingale square functions and square functions associated with $\{\widetilde{P}_t\}_{t>0}$.
	\begin{thm}\label{XZ}
		There exists a sequence $\{a_n\}_{n\geq 0}$ such that for any $1< p<\infty$,
		\begin{equation*}
			\|x\|_{p,\widetilde{P}}\approx \|x\|_{\mathcal{H}_p(\mathcal{M})}.
		\end{equation*}
	\end{thm}

\begin{corollary}\label{lower}
	Let $1<p<\infty$. Then 
	\begin{equation*}
		\alpha_p\gtrsim p\quad \text{and} \quad \beta_p\gtrsim \max\{p,p'\}.
	\end{equation*}
\end{corollary}
\begin{proof}
	By Theorem \ref{XZ}, we see $\alpha_p\gtrsim \gamma_p$ and $\beta_p\gtrsim \theta_p$. Then by means of Theorem \ref{JXR}, we immediately derive the desired lower bounds of $\alpha_p$ and $\beta_p$ in \eqref{nlps}.
	\end{proof}


\subsection{Col-sectorial and Row-sectorial operators}
The reader is referred to \cite[Chapter 3,\,4]{JX} for more details of this subsection. We first give the definition of sectorial operators.
\begin{definition}
	Let $1<p<\infty$, and let $A$ be a (possibly unbounded) linear operator $A$ on $L_p(\mathcal{M})$. For any $\omega\in (0,\pi)$, let 
	\begin{equation*}
		\Sigma_\omega =\{z\in\mathbb{C}^*: |\mathrm{Arg}(z)|<\omega\}
	\end{equation*}
be the open sector of angle $2\omega$ around the half-line $(0,+\infty)$. Denote by $\sigma(A)$ the spectrum set of $A$. $A$ is called a sectorial operator of type $\omega$ if $A$ is closed and densely defined, $\sigma(A)\subset \overline{\Sigma_\omega}$, and for any $\theta\in (\omega,\pi)$ there is a constant $K_\theta>0$ such that 
\begin{equation}\label{Ktheta}
	\|z(z-A)^{-1}\|\le K_\theta,\quad z\in\mathbb{C}\backslash \overline{\Sigma_\theta}.
\end{equation}

\end{definition}

\begin{rem}
	 Let $\{T_t\}_{t>0}$ be a noncommutative symmetric diffusion semigroup on $\mathcal{M}$. Let $1<p<\8$.  If $A$ is the corresponding negative infinitesimal generator of $\{T_t\}_{t>0}$ on $L_p(\M)$, then $A$ is closed and densely defined. Moreover, $\sigma(A)\subset \overline{\Sigma_{\frac{\pi}{2}}}$, and for any $z\in \mathbb{C}\backslash \overline{\Sigma_{\frac{\pi}{2}}}$, we have the following Laplace formula:
	 \begin{equation}\label{Laplace}
	 	(z-A)^{-1}=-\int_0^\infty e^{tz}T_t dt
	 \end{equation}
 in the strong operator topology. It is easy to show that $A$ is a sectorial operator of type $\frac{\pi}{2}$. 
\end{rem}
\begin{rem}\label{rem217}
	Let $1< p<\infty$. Note that $\sqrt{A}$ is the negative infinitesimal generator of the subordinated Poisson semigroup $\{P_t\}_{t>0}$ associated with the noncommutative symmetric diffusion semigroup $T_t=\exp(-t A)$ $({t>0})$ on $L_p(\M)$. Then $\sqrt{A}$ is a sectorial operator on $L_p(\mathcal{M})$ as well. Since $L_p(\mathcal{M})$ is reflexive, it follows from \cite[Theorem 3.8]{CD} that $L_p(\mathcal{M})$ has a direct sum decomposition
	\begin{equation}\label{oplus}
		L_p(\mathcal{M})=\overline{\mathrm{im}\sqrt{A}}\oplus \mathrm{ker}\sqrt{A}.
	\end{equation}
	Moreover, if $x\in \mathrm{ker}\sqrt{A}$, then $\|G_c^{P}(x)\|_{p}=\|G_r^{P}(x)\|_{p}=0$. Hence in the sequel, we always assume that $\sqrt{A}$ has dense range.
\end{rem}

We give a brief review of $H^\infty$ functional calculus which will be used later. For any $\theta\in (0,\pi)$, denote by $H^\infty(\Sigma_\theta)$ the space of all bounded analytic functions $f: \Sigma_\theta\to \mathbb{C}$. This is a Banach algebra for the norm
	\begin{equation*}
		\|f\|_{\infty,\theta}=\sup\{|f(z)|:z\in \Sigma_\theta\}.
	\end{equation*}
    Besides, denote by $H_0^\infty(\Sigma_\theta)$ the subalgebra of all $f\in H^\infty(\Sigma_\theta)$ for which there exist two positive numbers $s,c>0$ such that 
    \begin{equation}\label{festi}
	    |f(z)|\le c\frac{|z|^s}{(1+|z|)^{2s}},\quad z\in \Sigma_\theta.
    \end{equation}

Let $1<p<\infty$ and let $A$ be a sectorial operator of type $\omega\in (0,\pi)$ on $L_p(\mathcal{M})$. Let $\omega<\gamma<\theta<\pi$ and denote by $\Gamma_\gamma$ the oriented contour defined by 
\begin{equation*}
	\begin{aligned}
		\Gamma_\gamma(t)=\begin{cases}
			-te^{\mathrm{i}\gamma},& t\in\mathbb{R}_-,\\
			te^{-\mathrm{i}\gamma},& t\in\mathbb{R}_+.
		\end{cases}
	\end{aligned}
\end{equation*}
In other words, $\Gamma_\gamma$ is the boundary of $\Sigma_\gamma$ oriented counterclockwise. For any $f\in H_0^\infty(\Sigma_\theta)$, set
\begin{equation}\label{fA}
	f(A)=\frac{1}{2\pi\mathrm{i}}\int_{\Gamma_\gamma} f(z)(z-A)^{-1} dz.
\end{equation}
From \eqref{Ktheta} and \eqref{festi}, we know that this integral is absolutely convergent. Moreover, the integral in \eqref{fA} does not depend on the choice of $\gamma\in (\omega,\theta)$, and the mapping $f\mapsto f(A)$ is an algebra homomorphism from $H_0^\infty(\Sigma_{\theta})$ into $B(L_p(\mathcal{M}))$ which is consistent with the functional calculus of rational functions. 

The following resolution of the identity will be helpful. If $f\in H_0^\infty(\Sigma_\theta)$ satisfies
\begin{equation*}
	\int_0^\infty f(t)\frac{dt}{t}=1,
\end{equation*}
then one has
\begin{equation}\label{xfA}
	x=\int_0^\infty f(tA)(x)\frac{dt}{t},\quad \forall \ x\in\overline{\mathrm{im}A}=L_p(\M).
\end{equation}
We refer the reader to \cite{CD, MH1} for more details on functional calculus.

The following definition is vital in the proof of Theorem \ref{main1}.
  \begin{definition}
  	Let $1\le p<\infty$. A set $\mathcal{F}\subset B(L_p(\mathcal{M}))$ is said to be Col-bounded (resp. Row-bounded) if there is a constant $C>0$ such that for any finite families $T_1,\cdots,T_n$ in $\mathcal{F}$, and $x_1,\cdots,x_n$ in $L_p(\mathcal{M})$, we have
  	\begin{equation}\label{ColT}
  		\|\{T_k(x_k)\}_{k=1}^n\|_{L_p(\mathcal{M};\ell_2^c)}\le C\|\{x_k\}_{k=1}^n\|_{L_p(\mathcal{M};\ell_2^c)}
  	\end{equation}
  \begin{equation}\label{RowT}
  	\big(\text{resp.}\,\,\|\{T_k(x_k)\}_{k=1}^n\|_{L_p(\mathcal{M};\ell_2^r)}\le C\|\{x_k\}_{k=1}^n\|_{L_p(\mathcal{M};\ell_2^r)}.\big)
  \end{equation}
The best constant $C$ satisfying \eqref{ColT} (resp. \eqref{RowT}) is denoted by $\mathrm{Col}(\mathcal{F})$ (resp. $\mathrm{Row}(\mathcal{F})$).
  \end{definition}

The following useful lemma in terms of Col-boundedness and Row-boundedness will be needed later (see \cite[Lemma 4.2]{JX}).
\begin{lemma}\label{TFbound}
	Let $1\le p<\infty$. Let $\mathcal{F}\subset B(L_p(\mathcal{M}))$ be a set of bounded operators. Let $I$ be an interval of $\mathbb{R}$ and $C>0$ be a constant. Define
	\begin{equation*}
		\mathcal{T}=\bigg\{ \int_I f(t)S(t)dt \ \Big| \ S:I\to \mathcal{F}\,\, \text{is continuous},\,\, f\in L_1(I,dt),\,\, \text{and}\,\, \int_I|f(t)|dt\le C \bigg\}.
	\end{equation*}	
	\\
	(1) If $\mathcal{F}$ is Col-bounded, then $\mathcal{T}$ is Col-bounded with $\mathrm{Col}(\mathcal{T})\le C\mathrm{Col}(\mathcal{F})$,
	\\
	(2) if $\mathcal{F}$ is Row-bounded, then $\mathcal{T}$ is Row-bounded with $\mathrm{Row}(\mathcal{T})\le C\mathrm{Row}(\mathcal{F})$.
\end{lemma}

We will now introduce operators which are Col-sectorial (resp. Row-sectorial) of Col-type (resp. Row-type).
\begin{definition}
	Let $1<p<\infty$. An operator $A$ on $L_p(\mathcal{M})$ is said to be Col-sectorial (resp. Row-sectorial) of Col-type (resp. Row-type) $\omega$, if the set 
	\begin{equation*}
		\{z(z-A)^{-1}: z\in \mathbb{C}\backslash \overline{\Sigma_\theta}\}
	\end{equation*}
is Col-bounded (resp. Row-bounded) for any $\theta\in (\omega,\pi)$.
\end{definition}

\begin{lem}\label{Aalpha}
	Let $1<p<\infty$ and let $A$ be a sectorial operator of type $\omega\in (0,\pi)$ on $L_p(\mathcal{M})$. For any positive real number $\alpha>0$, let $A^\alpha$ be the corresponding fractional power of $A$. Assume that $\alpha\omega<\pi$. If $A$ is Col-sectorial (resp. Row-sectorial) of Col-type (resp. Row-type) $\omega$, then $A^\alpha$ is Col-sectorial (resp. Row-sectorial) of Col-type (resp. Row-type) $\alpha\omega$.
\end{lem}
\begin{proof}
	Let $\alpha \omega<\gamma<\pi$. We first prove that $\{z(z-A^\alpha)^{-1}:z\in \mathbb{C}\backslash\overline{\Sigma_\gamma}\}$ is Col-bounded. For any given $z\in \mathbb{C}\backslash\overline{\Sigma_\gamma}$, we decompose $z(z-A^\alpha)^{-1}$ into two parts:
	\begin{equation*}
		z(z-A^\alpha)^{-1}=z^{\frac{1}{\alpha}}(z^{\frac{1}{\alpha}}-A)^{-1}+\varphi(A),
	\end{equation*}
	where 
	\begin{equation*}
		\varphi(\lambda)=(z^{\frac{1}{\alpha}}\lambda^{\alpha}-z\lambda)(z-\lambda^{\alpha})^{-1}(z^{\frac{1}{\alpha}}-\lambda)^{-1}.
	\end{equation*}
Since $A$ is Col-sectorial of Col-type $\omega$, $\{z^{\frac{1}{\alpha}}(z^{\frac{1}{\alpha}}-A)^{-1}:z\in \mathbb{C}\backslash\overline{\Sigma_\gamma}\}$ is Col-bounded. Note that $\varphi$ is analytic in $\Sigma_{\frac{\gamma}{\alpha}}$. Let $\alpha\omega<\gamma'<\gamma$ and $\Gamma$ be the boundary of $\Sigma_{\frac{\gamma'}{\alpha}}$. Then
	\begin{equation*}
		\varphi(A)=\frac{1}{2\pi\mathrm{i}}\int_{\Gamma}\varphi(\lambda)(\lambda-A)^{-1}d\lambda.
	\end{equation*}
	Notice that the change of variable $\zeta=z^{-\frac{1}{\alpha}}\lambda$ yields
	\begin{equation*}
		\int_{\Gamma}|\varphi(\lambda)|\frac{|d\lambda|}{|\lambda|}\lesssim \frac{1}{\alpha}+\frac{1}{(\gamma-\gamma')^{2}}.
	\end{equation*}
	Hence using Lemma \ref{TFbound}, we have $\{\varphi(A):z\in \mathbb{C}\backslash\overline{\Sigma_\gamma}\}$ is also Col-bounded, which implies that $\{z(z-A^\alpha)^{-1}:z\in \mathbb{C}\backslash\overline{\Sigma_\gamma}\}$ is Col-bounded. Therefore $A^\alpha$ is Col-sectorial of Col-type $\alpha\omega$. In the same way, $A^\alpha$ is Row-sectorial of Row-type $\alpha\omega$.
\end{proof}

\subsection{Singular integral operators}
Let $\Delta$ be the diagonal of $\mathbb{R}^n\times \mathbb{R}^n$ and $H$ be a Hilbert space. We first give the definition of singular integral operators.
\begin{definition}
	An integral operator $T$ associated with a kernel $k:\mathbb{R}^{2n}\backslash\Delta\to H$ is called a singular integral operator, if for any smooth test function $f$ with compact support, we have
	\begin{equation*}
		Tf(x)=\int_{\mathbb{R}^n}k(x,y) f(y)dy,\quad \forall \ x\notin \mathrm{supp}f,
	\end{equation*}
    and the kernel satisfies size and smoothness conditions:
  \begin{equation}\label{smooth}
  	\begin{cases}
  		\|k(x,y)\|_H\le \frac{C}{|x-y|^n},\\
  		\|k(x,y)-k(x',y)\|_H+\|k(y,x)-k(y,x')\|_H\le \frac{C |x-x'|^\alpha}{|x-y|^{n+\alpha}}
  	\end{cases}
  \end{equation}
  for all $x,x',y\in\mathbb{R}^n$ with $|x-y|>2|x-x'|>0$ and some fixed $\alpha\in (0, 1]$ and constant $C>0$. 
\end{definition}

Let $\mathcal{M}$ be a von Neumann algebra equipped with a normal semifinite faithful trace $\tau$. Denote by $\A=L_\8(\mathbb{R}^n)\otimes \M$. Mei and Parcet proved the $L_p$-boundedness of $H$-valued singular integral operators in the semicommutative setting (see \cite[Theorem B2]{MP}).
\begin{thm}\label{singular}
	 Assume that the singular integral operator $T$ defines a bounded map from $L_2(\mathcal{A})$ to $L_2(\mathcal{A};H^r)=L_2(\mathcal{A};H^c)$. Then for $1<p<\infty$ and $f\in L_p(\mathcal{A})$, we have 
	\begin{equation*}
		\|Tf\|_{L_p(\mathcal{A};H^{r\cap c})}\le c_p\|f\|_{L_p(\mathcal{A})}, \quad\text{when}\,\, 2\le p<\infty,
	\end{equation*}
and
    \begin{equation*}
    	\|Tf\|_{L_p(\mathcal{A};H^{r+ c})}\le c_p\|f\|_{L_p(\mathcal{A})}, \quad\text{when}\,\, 1< p<2,
    \end{equation*}
where $c_p>0$ is the best relevant constants. Moreover, $c_p\approx p'$ as $p\to 1$ and $c_p\approx p$ as $p\to \infty$.
\end{thm}

\bigskip

\section{Estimates of $\alpha_p$ and $\beta_p$ where $2\le p<\infty$}\label{est1}

This section focuses on the proof of Theorem \ref{main1} where $2\le p<\infty$. We divide the proof into two parts. In the first, we deal with the translation group on $L_p(\mathbb{R},L_p(\mathcal{M}))$ for $2\le p<\infty$ and derive the corresponding semicommutative Littlewood-Paley-Stein inequality for its associated Poisson subordinated semigroup. In this case, we exploit techniques (i.e. Theorem \ref{singular}) from the semicommutative singular integral operator theory established in \cite{MP}. In the second part, by virtue of the dilatable property (i.e. Theorem \ref{thm210}), we transfer general noncommutative symmetric diffusion semigroups to this special case where $\{T_t\}_{t>0}$ is the translation group on $L_p(\mathbb{R},L_p(\mathcal{M}))$, and finally obtain an upper bound of noncommutative Littlewood-Paley-Stein inequalities for their associated Poisson subordinated semigroups.

\subsection{The case of the translation group}
 
 Let $\tau_t$ be the translation by $t$ on $L_p(\mathbb{R})$:
 \begin{equation*}
 	\tau_{t}(f)(s)=f(s+t),\quad \forall \ s\in\mathbb{R}.
 \end{equation*} 
Then $\{\tau_t\}_{t\in\mathbb{R}}$ is a strongly continuous group of positive isometries on $L_p(\mathbb{R})$. It is well-known that $\{\tau_t\}_{t\in\mathbb{R}}$ extends to a group of isometries on $L_p(\mathbb{R},L_p(\mathcal{M}))$ as well. Let $\{P_t^\tau\}_{t>0}$ be its associated Poisson subordinated semigroup. In the following, we aim to show that when $p\ge 2$,  for any $ f\in L_p(\mathbb{R},L_p(\mathcal{M}))$,
\begin{equation}\label{step1}
	\bigg\|\biggl(\int_0^\infty\big|t\partial P^\tau_t(f)\big|^2\frac{dt}{t}\biggr)^{1/2}\bigg\|_{L_p(\mathbb{R},L_p(\mathcal{M}))}\lesssim p \|f\|_{L_p(\mathbb{R},L_p(\mathcal{M}))},
\end{equation}
and
\begin{equation}\label{step111}
\bigg\|\biggl(\int_0^\infty\big|t\partial P^\tau_t(f)^*\big|^2\frac{dt}{t}\biggr)^{1/2}\bigg\|_{L_p(\mathbb{R},L_p(\mathcal{M}))}\lesssim p \|f\|_{L_p(\mathbb{R},L_p(\mathcal{M}))}.
\end{equation}

\subsubsection*{Step 1: Estimates of square functions associated with the translation group}	
 For any given $f\in L_p(\mathbb{R},L_p(\mathcal{M}))$, since $\tau_t$ is the translation by $t$ on $L_p(\mathbb{R},L_p(\mathcal{M}))$, we express $t\partial P_t^\tau$ as a convolution operator
\begin{equation*}
	\sqrt{t}\partial P_{\sqrt{t}}^\tau (f)(s)=\phi_t\ast f(s)=\int_{\mathbb{R}}\phi_t(s-r)f(r)dr,\quad \forall \ s\in\mathbb{R},\,\, t>0,
\end{equation*}
where $\phi$ is the function with Fourier transform
$$\quad \widehat{\phi}(\xi)=-\sqrt{-2 \pi \mathrm{i} \xi} e^{-\sqrt{-2 \pi \mathrm{i} \xi}}=-\sqrt{2 \pi|\xi|} e^{-\frac{i \operatorname{sgn}(\xi) \pi}{4}} \exp \left(-\sqrt{2 \pi|\xi|} e^{-\frac{i \operatorname{sgn}(\xi) \pi}{4}}\right), $$
and $\phi_t(x)=\frac{1}{t}\phi(\frac{x}{t})$ for $x\in\mathbb{R}$ and $t>0$ (see \cite[Section 7]{X}). Then
\begin{equation}\label{step11}
\begin{aligned}
	\biggl(\int_0^\infty\big|t\partial P^\tau_t(f)\big|^2\frac{dt}{t}\biggr)^{1/2}{}&=\frac{\sqrt{2}}{2}\biggl(\int_0^\infty\big|\sqrt{t}\partial P^\tau_{\sqrt{t}}(f)\big|^2\frac{dt}{t}\biggr)^{1/2}=\frac{\sqrt{2}}{2}\biggl(\int_0^\infty|\phi_t\ast f|^2\frac{dt}{t}\biggr)^{1/2}.
\end{aligned}
\end{equation}
This implies that
\begin{equation*}
	\begin{aligned}
		\bigg\|\biggl(\int_0^\infty\big|t\partial P^\tau_t(f)\big|^2\frac{dt}{t}\biggr)^{1/2}\bigg\|_{L_p(\mathbb{R},L_p(\mathcal{M}))}{}
		&=\frac{\sqrt{2}}{2}\|\{\phi_{t}\ast f\}_{t>0}\|_{L_p(L_\infty(\mathbb{R})\otimes\mathcal{M};(L_2(\mathbb{R}_+,\frac{dt}{t}))^c)}.
	\end{aligned}
\end{equation*}
Let $H=L_2(\mathbb{R}_+,\frac{dt}{t})$ and define $$Tf=\{\phi_t\ast f\}_{t>0}, \quad \forall \ f\in L_p(\mathbb{R},L_p(\mathcal{M})). $$
In \cite[Section 7]{X}, it has been shown that
\begin{equation*}
	|\phi(x)|\lesssim \frac{1}{(1+|x|)^{3/2}} \quad \text{and} \quad |\phi'(x)|\lesssim \frac{1}{(1+|x|)^{5/2}},\quad \forall \ x\in\mathbb{R}.
\end{equation*}
By means of the above inequalities, we verify that $\{\phi_t\}_{t>0}$ satisfies \eqref{smooth}. Note also that $T$ is a bounded map from $L_2(\mathbb{R},L_2(\mathcal{M}))$ to $L_2(L_\infty(\mathbb{R})\otimes\mathcal{M};(L_2(\mathbb{R}_+,\frac{dt}{t}))^c)$. Hence from Theorem \ref{singular}, one has for $p\geq 2$
\begin{equation*}
	\|\{\phi_{t}\ast f\}_{t>0}\|_{L_p(L_\infty(\mathbb{R})\otimes\mathcal{M};(L_2(\mathbb{R}_+,\frac{dt}{t}))^{r+c})}\lesssim p\|f\|_{L_p(\mathbb{R},L_p(\mathcal{M}))}.
\end{equation*}
As a result, we obtain
\begin{equation*}
	\begin{aligned}
		\bigg\|\biggl(\int_0^\infty\big|t\partial P^\tau_t(f)\big|^2\frac{dt}{t}\biggr)^{1/2}\bigg\|_{L_p(\mathbb{R},L_p(\mathcal{M}))}+\bigg\|\biggl(\int_0^\infty\big|t\partial P^\tau_t(f)^*\big|^2\frac{dt}{t}\biggr)^{1/2}\bigg\|_{L_p(\mathbb{R},L_p(\mathcal{M}))}\lesssim p\|f\|_{L_p(\mathbb{R},L_p(\mathcal{M}))}.
	\end{aligned}
\end{equation*}
This finishes the proof of (\ref{step1}) and (\ref{step111}).

	
\subsection{The case of noncommutative symmetric diffusion semigroups}
\subsubsection*{Step 2: Transference}	
We utilize the transference method to transfer the general case to the special one of the translation group. To this end, we first need to dilate our semigroup $\{T_t\}_{t>0}$ to a group of invertible isometries. More precisely, since $\{T_t\}_{t>0}$ is a noncommutative symmetric diffusion semigroup, by Theorem \ref{thm210} we see that for any $1<p<\8$, there exists a noncommutative $L^p$-space $L^p(\N)$, a strongly continuous group of completely positive invertible isometries $U_t: L^p(\N) \rightarrow L^p(\N)$ and two completely positive contractions $J: L^p(\M) \rightarrow L^p(\N)$ and $P: L^p(\N) \rightarrow L^p(\M)$ such that
$$
T_t=P U_t J, \quad \forall \ t > 0.
$$
By the above dilation, we can assume that $\{T_t\}_{t>0}$ itself is a group of completely positive invertible isometries on $L_p(\mathcal{M})$ for $1< p<\infty$.
Recall that for any $x\in L_p(\mathcal{M})$,
\[G_c^{P}(x)=\biggl(\int_0^\infty \big|t\partial P_t(x)\big|^2\frac{dt}{t}\biggr)^{1/2} \quad\text{and}\quad G_r^{P}(x)=\biggl(\int_0^\infty \big|t\partial P_t(x)^*\big|^2\frac{dt}{t}\biggr)^{1/2}.\]
We first prove that for $p\geq 2$
\[\|G_c^{P}(x)\|_{p}\lesssim p\|x\|_p, \quad\forall \ x\in L_p(\mathcal{M}).\]
Let $x\in L_p(\mathcal{M})$ with $\|x\|_p=1$.
Denote the ergodic averages of $\{T_t\}_{t>0}$ by
\begin{equation*}
	M_t=\frac{1}{t}\int_0^t T_s ds,\quad \forall \ t>0, 
\end{equation*}
and $\{M_t^\tau\}_{t>0}$ are the corresponding ergodic averages of the translation group $\{\tau_t\}_{t>0}$.
From \eqref{Ptdef}, we write
\begin{equation*}
	P_t=\frac{1}{2\sqrt{\pi}}\int_0^\infty\frac{t}{s^{3/2}}\exp\big(-\frac{t^2}{4s}\big)T_s ds,\quad \forall \ t>0.
\end{equation*}
This implies that
\begin{equation}\label{tPt}
	\begin{aligned}
		t\partial P_t{}&=\frac{1}{2\sqrt{\pi}}\int_0^\infty\big(\frac{t}{s^{3/2}}-\frac{t^3}{2s^{5/2}}\big)\exp\big(-\frac{t^2}{4s}\big)T_s ds\\
		&=\frac{1}{2\sqrt{\pi}}\int_0^\infty\big(\frac{t}{s^{3/2}}-\frac{t^3}{2s^{5/2}}\big)\exp\big(-\frac{t^2}{4s}\big)(sM_s)' ds\\
		&=\int_0^\infty \varphi(\frac{t}{\sqrt{s}})M_s \frac{ds}{s}\\
		&=\int_0^\infty \varphi(\frac{1}{\sqrt{s}})M_{t^2s} \frac{ds}{s},
	\end{aligned}
\end{equation}
where
\[\varphi(r)=\frac{1}{16\sqrt{\pi}}(12r-12r^3+r^5)e^{-r^2/4},\quad \forall \ r>0.\]
Note that $\varphi(r)$ decays rapidly as $r\rightarrow \8$. Then there exists $N>0$ and $C>0$ such that for any $a>N$,
\begin{equation}\label{intphi}
	\int_a^\infty\Big|\varphi(\frac{1}{\sqrt{s}})\Big|\frac{ds}{s}\le \frac{C}{\sqrt{a}}.
\end{equation}
In the following, we assume that $a$ is large enough to satisfy (\ref{intphi}). By the triangle inequality, one has
	\begin{equation*}
		\begin{aligned}
			\|G_c^{P}(x)\|_{p}{}
			&=\biggl\|\biggl(\int_0^\infty \bigg|\int_0^\infty\varphi(\frac{1}{\sqrt{s}})M_{t^2s}(x)\frac{ds}{s}\bigg|^2\frac{dt}{t}\biggr)^{1/2}\biggr\|_{p}\\
			&\le \lim_{b\to\infty}\biggl\|\biggl(\int_{b^{-1}}^b \bigg|\int_0^a\varphi(\frac{1}{\sqrt{s}})M_{t^2s}(x)\frac{ds}{s}\bigg|^2\frac{dt}{t}\biggr)^{1/2}\biggr\|_{p}\\
			&\quad\quad+\lim_{b\to\infty}\biggl\|\biggl(\int_{b^{-1}}^b \bigg|\int_a^\infty\varphi(\frac{1}{\sqrt{s}})M_{t^2s}(x)\frac{ds}{s}\bigg|^2\frac{dt}{t}\biggr)^{1/2}\biggr\|_{p}.
		\end{aligned}
	\end{equation*}
	Denote by
	\begin{equation}\label{caseI}
		\mathrm{I}_{b, a}=\biggl\|\biggl(\int_{b^{-1}}^b \bigg|\int_0^a\varphi(\frac{1}{\sqrt{s}})M_{t^2s}(x)\frac{ds}{s}\bigg|^2\frac{dt}{t}\biggr)^{1/2}\biggr\|_{p}
	\end{equation}
	and
	\begin{equation}\label{caseII}
		\mathrm{II}_{b, a}=\biggl\|\biggl(\int_{b^{-1}}^b \bigg|\int_a^\infty\varphi(\frac{1}{\sqrt{s}})M_{t^2s}(x)\frac{ds}{s}\bigg|^2\frac{dt}{t}\biggr)^{1/2}\biggr\|_{p},
	\end{equation}
	where $b$ is a positive constant large enough. We are about to estimate $\text{I}_{b, a}$ and $\text{II}_{b, a}$ respectively. We start with estimating $\text{II}_{b, a}$.
	
\subsubsection*{Step 3: Estimate of \,$\mathrm{II}_{b, a}$}	
	Let
	\[M^*(x)=\sup_{v>0}\|M_v(x)\|_{p}.\]
	Since $T_t$ ($t>0$) is an isometry, by the triangle inequality, we obtain
	\begin{equation*}
		\begin{aligned}
			M^*(x)\le \sup_{v>0}\frac{1}{v}\int_0^v \|T_s(x)\|_{p}ds=\|x\|_{p}=1.
		\end{aligned}
	\end{equation*}
 From \eqref{intphi}, we deduce
	\begin{equation}\label{case1canp}
		\biggl\|\int_a^\infty \varphi(\frac{1}{\sqrt{s}})M_{t^2s}(x)\frac{ds}{s}\biggr\|_{p}\le \frac{C}{\sqrt{a}}M^*(x)\le \frac{C}{\sqrt{a}}.
	\end{equation}
Hence one has
	\begin{equation}\label{IIsum}
		\begin{aligned}
			\mathrm{II}_{b, a}{}&=\biggl\|\int_{b^{-1}}^b \bigg|\int_a^\infty\varphi(\frac{1}{\sqrt{s}})M_{t^2s}(x)\frac{ds}{s}\bigg|^2\frac{dt}{t}\biggr\|^{1/2}_{p/2}\\
			&\le \biggl(\int_{b^{-1}}^b \bigg\|\int_a^\infty\varphi(\frac{1}{\sqrt{s}})M_{t^2s}(x)\frac{ds}{s}\bigg\|^2_{p}\frac{dt}{t}\biggr)^{1/2}\le \frac{(2\ln b)^{1/2}C}{\sqrt{a}}.
		\end{aligned}
	\end{equation}

\subsubsection*{Step 4: Estimate of \,$\mathrm{I}_{b, a}$}	
Note that $\{T_t\}_{t>0}$ is a group of completely positive invertible isometries on $L_p(\mathcal{M})$. This implies that $T_t\otimes \text{Id}_{L_2(\mathbb{R}_+,\frac{dt}{t})}$ ($t>0$) is  isometric on $L_p(\mathcal{M};(L_2(\mathbb{R}_+,\frac{dt}{t}))^c)$. Let $c>0$ be a constant. Thus for any $0<u\le c$, one has
	\begin{equation*}
		\begin{aligned}
			\mathrm{I}_{b, a}{}&=\biggl\|\biggl(\int_{b^{-1}}^b \bigg|\int_0^a\varphi(\frac{1}{\sqrt{s}})M_{t^2s}(x)\frac{ds}{s}\bigg|^2\frac{dt}{t}\biggr)^{1/2}\biggr\|_{p}\\
			&=\biggl\|\bigg\{\int_0^a\varphi(\frac{1}{\sqrt{s}})\mathbbm{1}_{[b^{-1},b]}(t)M_{t^2s}(x)\frac{ds}{s}\bigg\}_{t>0}\biggr\|_{L_p(\mathcal{M};(L_2(\mathbb{R}_+,\frac{dt}{t}))^c)}\\
			&=\biggl\|\bigg\{\int_0^a\varphi(\frac{1}{\sqrt{s}})\mathbbm{1}_{[b^{-1},b]}(t)T_uM_{t^2s}(x)\frac{ds}{s}\bigg\}_{t>0}\biggr\|_{L_p(\mathcal{M};(L_2(\mathbb{R}_+,\frac{dt}{t}))^c)}\\
			&=\biggl\|\biggl(\int_{b^{-1}}^b \bigg|\int_0^a\varphi(\frac{1}{\sqrt{s}})M_{t^2s}T_u(x)\frac{ds}{s}\bigg|^2\frac{dt}{t}\biggr)^{1/2}\biggr\|_{p}.
		\end{aligned}
	\end{equation*}
 Define $g_{a,b,c}: \mathbb{R}\to L_p(\mathcal{M})$ given by 
	\begin{equation*}
		g_{a,b,c}(u)=\chi_{(0,ab^2+c]}(u)T_u(x).
	\end{equation*}
	One can check that
	\begin{equation*}
		M_{t^2s}T_u(x)=M_{t^2s}^\tau(g_{a,b,c})(u),\,\,\,\,\,\,\,\, \forall \ 0<s \le a,\,\,0<t\le b,\,\,0<u\le c.
	\end{equation*}
	Hence we get
	\begin{equation*}
		\begin{aligned}
			\mathrm{I}_{b, a}^p{}&=\frac{1}{c}\int_0^c \biggl\|\biggl(\int_{b^{-1}}^b \bigg|\int_0^a\varphi(\frac{1}{\sqrt{s}})M^\tau_{t^2s}(g_{a,b,c})(u)\frac{ds}{s}\bigg|^2\frac{dt}{t}\biggr)^{1/2}\biggr\|^{p}_{p}du\\
			&\le \frac{1}{c}\int_{-\infty}^\infty \biggl\|\biggl(\int_{b^{-1}}^b \bigg|\int_0^a\varphi(\frac{1}{\sqrt{s}})M^\tau_{t^2s}(g_{a,b,c})(u)\frac{ds}{s}\bigg|^2\frac{dt}{t}\biggr)^{1/2}\biggr\|^{p}_{p}du\\
			&=\frac{1}{c}\biggl\|\biggl(\int_{b^{-1}}^b \bigg|\int_0^a\varphi(\frac{1}{\sqrt{s}})M^\tau_{t^2s}(g_{a,b,c})\frac{ds}{s}\bigg|^2\frac{dt}{t}\biggr)^{1/2}\biggr\|^p_{L_p(\mathbb{R},L_p(\mathcal{M}))}.
		\end{aligned}
	\end{equation*}
This implies that
	\begin{equation}\label{Iequal}
		\begin{aligned}
			\mathrm{I}_{b, a}{}&\le c^{-\frac{1}{p}}\biggl\|\biggl(\int_{b^{-1}}^b \bigg|\int_0^a\varphi(\frac{1}{\sqrt{s}})M^\tau_{t^2s}(g_{a,b,c})\frac{ds}{s}\bigg|^2\frac{dt}{t}\biggr)^{1/2}\biggr\|_{L_p(\mathbb{R},L_p(\mathcal{M}))}\\
			&\le c^{-\frac{1}{p}}\biggl\|\biggl(\int_{b^{-1}}^b \bigg|\int_0^\infty\varphi(\frac{1}{\sqrt{s}})M^\tau_{t^2s}(g_{a,b,c})\frac{ds}{s}\bigg|^2\frac{dt}{t}\biggr)^{1/2}\biggr\|_{L_p(\mathbb{R},L_p(\mathcal{M}))}\\
			&\,\,\,\,\,\,+c^{-\frac{1}{p}}\biggl\|\biggl(\int_{b^{-1}}^b \bigg|\int_a^\infty\varphi(\frac{1}{\sqrt{s}})M^\tau_{t^2s}(g_{a,b,c})\frac{ds}{s}\bigg|^2\frac{dt}{t}\biggr)^{1/2}\biggr\|_{L_p(\mathbb{R},L_p(\mathcal{M}))}.
		\end{aligned}
	\end{equation}
	Let
	\[(M^\tau)^*(g_{a,b,c})=\sup_{v>0}\|M_v^\tau(g_{a,b,c})\|_{p}.\]
	From \cite{AMM} or \cite[Theorem 5.2.5]{UK}, we have
	\begin{equation}\label{Mtao1}
		\|(M^\tau)^*(g_{a,b,c})\|_{L_p(\mathbb{R})}\le p^\prime\|g_{a,b,c}\|_{L_p(\mathbb{R},L_p(\mathcal{M}))}=p^\prime(ab^2+c)^{1/p}.
	\end{equation}
	Thus by \eqref{intphi}, \eqref{Iequal} and \eqref{Mtao1}, we deduce that
	\[\mathrm{I}_{b, a}\le c^{-1/p}\biggl\|\biggl(\int_{b^{-1}}^b \bigg|\int_0^\infty\varphi(\frac{1}{\sqrt{s}})M^\tau_{t^2s}(g_{a,b,c})\frac{ds}{s}\bigg|^2\frac{dt}{t}\biggr)^{1/2}\biggr\|_{L_{p}(\mathbb{R},L_p(\mathcal{M}))}+\frac{(2\ln b)^{1/2}Cp^\prime}{\sqrt{a}}\biggl(\frac{ab^2+c}{c}\biggr)^{1/p}.\]
	Furthermore, by \eqref{step1} one has
	\[\biggl\|\biggl(\int_{b^{-1}}^b \bigg|\int_0^\infty\varphi(\frac{1}{\sqrt{s}})M^\tau_{t^2s}(g_{a,b,c})\frac{ds}{s}\bigg|^2\frac{dt}{t}\biggr)^{1/2}\biggr\|_{L_{p}(\mathbb{R},L_p(\mathcal{M}))}\lesssim p\|g_{a,b,c}\|_{L_p(\mathbb{R},L_p(\mathcal{M}))}=p(ab^2+c)^{1/p},\]
	which implies that
	\begin{equation}\label{Isum}
		\mathrm{I}_{b, a}\lesssim p\biggl(\frac{ab^2+c}{c}\biggr)^{1/p}+\frac{(2\ln b)^{1/2}Cp^\prime}{\sqrt{a}}\biggl(\frac{ab^2+c}{c}\biggr)^{1/p}.
	\end{equation}
	Therefore, from \eqref{IIsum} and \eqref{Isum} we have
	\begin{equation*}
		\begin{aligned}
			\mathrm{I}_{b, a}+\mathrm{II}_{b, a}{}&\lesssim p\biggl(\frac{ab^2+c}{c}\biggr)^{1/p}+\frac{(2\ln b)^{1/2}Cp^\prime}{\sqrt{a}}\biggl(\frac{ab^2+c}{c}\biggr)^{1/p}+\frac{(2\ln b)^{1/2}C}{\sqrt{a}}\\
			&\lesssim p\biggl(\frac{ab^2+c}{c}\biggr)^{1/p}+\frac{(2\ln b)^{1/2}C}{\sqrt{a}}\biggl(1+p^\prime\biggl(\frac{ab^2+c}{c}\biggr)^{1/p}\biggr).
		\end{aligned}
	\end{equation*}
	Let successively $c\to\infty$, $a\to\infty$ and $b\to\infty$, then
	\[\|G_c^{P}(x)\|_{p}\lesssim p.\]
Similarly, in the same way, one can show that
\begin{equation*}
	\|G_r^{P}(x)\|_{p}\lesssim p.
\end{equation*}
Therefore, we conclude that for any $x\in L_p({\mathcal{M}})$
\begin{equation*}
	\|x\|_{p,P}\lesssim p\|x\|_{p}.
\end{equation*}
This yields that $\beta_p\lesssim p$ as $p\to\8$. However, from Corollary \ref{lower} we get that $\beta_p\gtrsim p$ as $p\to\infty$. As a result,
$p$ is the optimal order of $\beta_p$ as $p\to\infty$.

\subsubsection*{Step 5: Duality}	
Let $F(z)=-ze^{-z}$. Let $x\in L_p(\mathcal{M})$ and $y\in L_{p'}(\mathcal{M})$.
It follows from (\ref{ppppp}) that
\begin{equation*}
	[y]_{p',P}\lesssim p^{2}\|y\|_{p'}.
\end{equation*}
Thus there exist two operator-valued functions $u_1,u_2$ such that $u_1(t)+u_2(t)=t\partial P_t(y)=F(t\sqrt{A})(y)$ $(\forall \ t\in\mathbb{R}_+)$, which satisfy
\begin{equation}\label{yy1y2}
	 \|u_1\|_{L_{p'}(\mathcal{M};(L_2(\mathbb{R}_+,\frac{dt}{t}))^c)}+\|u_2\|_{L_{p'}(\mathcal{M};(L_2(\mathbb{R}_+,\frac{dt}{t}))^r)} \lesssim p^{2}\|y\|_{p'}.
\end{equation}
 From \eqref{xfA}, one has
\begin{equation*}
	x=4\int_0^\infty F(t\sqrt{A})F(t\sqrt{A})(x)\frac{dt}{t}.
\end{equation*}
Then by the duality in \eqref{duality1} and \eqref{yy1y2}, we deduce
\begin{equation*}
	\begin{aligned}
		|\langle x,y\rangle|{}&=4\bigg|\int_0^\infty \langle F(t\sqrt{A})(x), F(t\sqrt{A})(y)\rangle \frac{dt}{t}\bigg|\\
		&\le 4\bigg|\int_0^\infty \langle F(t\sqrt{A})(x), u_1(t)\rangle \frac{dt}{t}\bigg|+4\bigg|\int_0^\infty \langle F(t\sqrt{A})(x), u_2(t)\rangle \frac{dt}{t}\bigg|\\
		&\le 4\big(\|G_r^{P}(x)\|_{p}\|u_1\|_{L_{p'}(\mathcal{M};(L_2(\mathbb{R}_+,\frac{dt}{t}))^c)}+\|G_c^{P}(x)\|_{p}\|u_2\|_{L_{p'}(\mathcal{M};(L_2(\mathbb{R}_+,\frac{dt}{t}))^r)}\big)\\
		&\lesssim p^{2}\|x\|_{p,P}\|y\|_{p'}.
	\end{aligned}
\end{equation*}
Hence this implies that
\begin{equation*}
	\|x\|_{p}\lesssim p^{2}\|x\|_{p,P}, \quad \forall \ x\in L_p(\mathcal{M}),
\end{equation*}
and
\begin{equation*}
	\alpha_p\lesssim p^{2}.
\end{equation*}

\bigskip

\section{Estimates of $\alpha_p,\widetilde{\alpha}_p$ and $\beta_p$ where $1<p<2$}\label{est2}
 Recall that the ergodic averages of $\{T_t\}_{t>0}$ are given by
\begin{equation*}
	M_t=\frac{1}{t}\int_0^t T_s ds,\quad \forall \ t>0.
\end{equation*}
The following lemma describes the Col-boundedness and Row-boundedness of $\{M_t\}_{t>0}$.
\begin{lemma}\label{Mtbound}
	Let $1< p<\infty$. The family $\{M_t\}_{t>0}$ is Col-bounded and Row-bounded on $L_p(\mathcal{M})$. Moreover, 
	$$\mathrm{Col}(\{M_t\}_{t>0})\lesssim \max\{p,p'\} \quad \text{and} \quad \mathrm{Row}(\{M_t\}_{t>0})\lesssim \max\{p,p'\}. $$
\end{lemma}

\begin{proof}
	When $p=2$, for any $t_1,\cdots,t_n>0$ and $x_1,\cdots,x_n\in L_2(\mathcal{M})$, one has
  \begin{equation*}
  	\begin{aligned}
  		\bigg\|\biggl(\sum_{k=1}^n |M_{t_k}(x_k)|^2\biggr)^{1/2}\bigg\|_{2}{}
  	&=\biggl(\sum_{k=1}^n \|M_{t_k}(x_k)\|^2_{2}\biggr)^{1/2}\\
  	&\le \biggl(\sum_{k=1}^n \|x_k\|^2_{2}\biggr)^{1/2}=\bigg\|\biggl(\sum_{k=1}^n |x_k|^2\biggr)^{1/2}\bigg\|_{2}.
  	\end{aligned}
  \end{equation*}
This implies that
 \begin{equation*}
 	\|\{M_{t_k}(x_k)\}_{k=1}^n\|_{L_2(\mathcal{M};\ell_2^c)}\le \|\{x_k\}_{k=1}^n\|_{L_2(\mathcal{M};\ell_2^c)}.
 \end{equation*}
In the same way, we also get
\begin{equation}\label{dual22}
	\|\{M_{t_k}(x_k)\}_{k=1}^n\|_{L_2(\mathcal{M};\ell_2^r)}\le \|\{x_k\}_{k=1}^n\|_{L_2(\mathcal{M};\ell_2^r)}.
\end{equation}

Suppose $1<p<2$, we let $q=2(p'-1)$. Since $q>2$, from \cite[Theorem 4.5]{JX2}, we have
\begin{equation*}
	\|\{M_{t_k}(x_k)\}_{k=1}^n\|_{L_{\frac{q}{q-2}}(\mathcal{M};\ell_\infty)}\lesssim q^2\|\{x_k\}_{k=1}^n\|_{L_{\frac{q}{q-2}}(\mathcal{M};\ell_\infty)}.
\end{equation*}
Note that each $T_t$ is a unital completely positive and selfadjoint map on $\mathcal{M}$. Then by \cite[Lemma 2.4]{GSX}, we obtain
\begin{equation*}
	\|\{M_{t_k}(x_k)\}_{k=1}^n\|_{L_{\frac{q}{q-1}}(\mathcal{M};\ell_2^c)}\lesssim q\|\{x_k\}_{k=1}^n\|_{L_{\frac{q}{q-1}}(\mathcal{M};\ell_2^c)}.
\end{equation*}
By the duality in \eqref{duality2}, one has 
\begin{equation}\label{dualqq}
	\|\{M_{t_k}(x_k)\}_{k=1}^n\|_{L_{q}(\mathcal{M};\ell_2^r)}\lesssim q\|\{x_k\}_{k=1}^n\|_{L_{q}(\mathcal{M};\ell_2^r)}.
\end{equation}
Hence by the complex interpolation in \eqref{compinter}, we get from \eqref{dual22} and \eqref{dualqq} that
\begin{equation}\label{p'p'}
		\|\{M_{t_k}(x_k)\}_{k=1}^n\|_{L_{p'}(\mathcal{M};\ell_2^r)}\lesssim q^{1-\alpha}\|\{x_k\}_{k=1}^n\|_{L_{p'}(\mathcal{M};\ell_2^r)},
\end{equation}
where $\alpha$ satisfies $\frac{1-\alpha}{q}+\frac{\alpha}{2}=\frac{1}{p'}$. Note that $$q^{1-\alpha}\le q=2(p'-1).$$ 
Then we deduce that
\begin{equation*}
	\|\{M_{t_k}(x_k)\}_{k=1}^n\|_{L_{p'}(\mathcal{M};\ell_2^r)}\lesssim p'\|\{x_k\}_{k=1}^n\|_{L_{p'}(\mathcal{M};\ell_2^r)}.
\end{equation*}
By duality again, one has
\begin{equation*}
	\|\{M_{t_k}(x_k)\}_{k=1}^n\|_{L_{p}(\mathcal{M};\ell_2^c)}\lesssim p'\|\{x_k\}_{k=1}^n\|_{L_{p}(\mathcal{M};\ell_2^c)}.
\end{equation*}
Similarly, we obtain
\begin{equation*}
	\|\{M_{t_k}(x_k)\}_{k=1}^n\|_{L_{p'}(\mathcal{M};\ell_2^c)}\lesssim p'\|\{x_k\}_{k=1}^n\|_{L_{p'}(\mathcal{M};\ell_2^c)}
\end{equation*}
and
\begin{equation*}
		\|\{M_{t_k}(x_k)\}_{k=1}^n\|_{L_{p}(\mathcal{M};\ell_2^r)}\lesssim p'\|\{x_k\}_{k=1}^n\|_{L_{p}(\mathcal{M};\ell_2^r)}.
\end{equation*}
This finishes the proof.
\end{proof}

The following proposition shows that the negative generator $\sqrt{A}$ associated with $\{P_t\}_{t>0}$ is Col-sectorial of Col-type $\frac{\pi}{4}$ and Row-sectorial of Row-type $\frac{\pi}{4}$ when $1<p<\infty$.
\begin{proposition}\label{mainlemma}
	Let $1<p<\infty$. Suppose that $\{T_t\}_{t>0}$ is a noncommutative symmetric diffusion semigroup on $\mathcal{M}$ and $\{P_t\}_{t>0}$ is its associated subordinated Poisson semigroup. Then the negative infinitesimal generator $A$ of $\{T_t\}_{t>0}$ is Col-sectorial of Col-type $\frac{\pi}{2}$ and Row-sectorial of Row-type $\frac{\pi}{2}$. More precisely, the family 
	\begin{equation*}
		\{z(z-A)^{-1}:z\in \mathbb{C}\backslash \overline{\Sigma_\theta}\}
	\end{equation*}
    is Col-bounded and Row-bounded on $L_p(\mathcal{M})$ with constant $C_{\theta}\max\{p,p'\}$ for any $\frac{\pi}{2}<\theta<\pi$, where $C_{\theta}$ is a constant depending on $\theta$.
    
    Consequently, 
	$\sqrt{A}$ is Col-sectorial of Col-type $\frac{\pi}{4}$ and Row-sectorial of Row-type $\frac{\pi}{4}$. More precisely, the family 
	\begin{equation*}
		\{z(z-\sqrt{A})^{-1}:z\in \mathbb{C}\backslash \overline{\Sigma_\theta}\}
	\end{equation*}
	is Col-bounded and Row-bounded on $L_p(\mathcal{M})$ with constant $c_{\theta}\max\{p,p'\}$ for any $\frac{\pi}{4}<\theta<\pi$, where $c_{\theta}$ is a constant depending on $\theta$.
\end{proposition}
\begin{proof}
	Let $z\in\mathbb{C}$ with $\mathrm{Re}z<0$. From the Laplace formula in \eqref{Laplace} we calculate that 
	\begin{equation*}
		\begin{aligned}
			(z-A)^{-1}=-\int_0^\infty e^{tz}T_t dt=-\int_0^\infty e^{tz}(tM_t)' dt=\int_0^\infty tze^{tz}M_t dt.
		\end{aligned}
	\end{equation*}
In Lemma \ref{Mtbound}, we have shown that $\{M_t\}_{t>0}$ is Col-bounded and Row-bounded on $L_p(\mathcal{M})$. Note that 
\begin{equation*}
	\begin{aligned}
		\int_0^\infty \bigg|\frac{(\mathrm{Re}z)^2}{|z|}zte^{tz}\bigg|dt=1.
	\end{aligned}
\end{equation*}
From Lemma \ref{TFbound}, this implies that $\big\{ \frac{(\mathrm{Re}z)^2}{|z|} (z-A)^{-1}:\mathrm{Re}z<0 \big\}$ is also Col-bounded and Row-bounded on $L_p(\mathcal{M})$. Assume that $\frac{\pi}{2}<\theta<\pi$. Notice that $$\frac{|z|}{\mathrm{Re}(z)}\le \frac{1}{\mathrm{cos}(\pi-\theta)},\quad \forall \ z\in \mathbb{C}\backslash \overline{\Sigma_\theta}.$$ Thus for any finite numbers $z_1,\cdots,z_n$ in $\mathbb{C}\backslash \overline{\Sigma_\theta}$ and $x_1,\cdots,x_n$ in $L_p(\mathcal{M})$, we have
\begin{equation*}
	\bigg\| \biggl(\sum_{k=1}^n |z_k(z_k-A)^{-1}(x_k)|^2 \biggr)^{1/2}\bigg\|_{p}\le \frac{1}{\mathrm{cos}^2(\pi-\theta)}\bigg\| \biggl(\sum_{k=1}^n \bigg|\frac{(\mathrm{Re}z_k)^2}{|z_k|}(z_k-A)^{-1}(x_k)\bigg|^2 \biggr)^{1/2}\bigg\|_{p}.
\end{equation*}
Hence $\{ z(z-A)^{-1}:z\in \mathbb{C}\backslash \overline{\Sigma_\theta}\}$ is Col-bounded on $L_p(\mathcal{M})$, and is also Row-bounded in the same way. The second part on $\sqrt{A}$ is derived by Lemma \ref{Aalpha} directly if we let $\alpha=\frac{1}{2}$.

\end{proof}

Now we begin to estimate $\alpha_p$ and $\beta_p$ for $1<p<2$.
\begin{proof}[Estimates of $\alpha_p$ and $\beta_p$]
	Let $1<p<2$. By Proposition \ref{mainlemma}, $\sqrt{A}$ is Col-sectorial of Col-type $\frac{\pi}{4}$ and Row-sectorial of Row-type $\frac{\pi}{4}$. Let $\xi=\frac{3\pi}{8}$.
	 Let $F(z)=ze^{-z}$, $G(z)=4F(z)$ and $\widetilde{G}(z)=\overline{G(\overline{z})}$. It is easy to check that $F,G\in H_0^\infty(\Sigma_{\xi})\backslash \{0\}$ and $$\int_0^\infty G(t)F(t)dt=1.$$ 
 Hence repeating the proof of \cite[Theorem 7.8]{JX}, we get
	\begin{equation}\label{xQcQr}
		\|x\|_{p,P}\le 2\max\{\|Q_c\|,\|Q_r\|\}[x]_{p,P},
	\end{equation}
where $Q_c$ and $Q_r$ are two operators defined as
\begin{equation*}
	\begin{aligned}
		Q_c: L_p(\mathcal{M};(L_2(\mathbb{R}_+,\frac{dt}{t}))^c)&\to L_p(\mathcal{M};(L_2(\mathbb{R}_+,\frac{dt}{t}))^c)\\
		\{x_t\}_{t>0}&\mapsto \biggl\{\int_0^\infty F(sA)G(tA)x_t
		\frac{dt}{t}\biggr\}_{s>0}
	\end{aligned}
\end{equation*}
and
\begin{equation*}
	\begin{aligned}
		Q_r: L_p(\mathcal{M};(L_2(\mathbb{R}_+,\frac{dt}{t}))^r)&\to L_p(\mathcal{M};(L_2(\mathbb{R}_+,\frac{dt}{t}))^r)\\
		\{x_t\}_{t>0}&\mapsto \biggl\{\int_0^\infty F(sA)G(tA)x_t
		\frac{dt}{t}\biggr\}_{s>0}.
	\end{aligned}
\end{equation*}
	It suffices to estimate $\|Q_c\|$ and $\|Q_r\|$. Let $\frac{\pi}{4}<\eta<\xi$. Let $\Gamma_\eta$ be the oriented contour defined by $\Gamma_\eta=\{f_\eta(t):t\in\mathbb{R}\}$, where
	\begin{equation*}
		\begin{aligned}
			f_\eta(t)=\begin{cases}
				-te^{\mathrm{i}\eta},& t\in\mathbb{R}_-,\\
				te^{-\mathrm{i}\eta},& t\in\mathbb{R}_+.
			\end{cases}
		\end{aligned}
	\end{equation*}
	Thus $\Gamma_\eta$ is the boundary of $\Sigma_\eta$ oriented counterclockwise. 
	Then let 
	\begin{equation*}
		\begin{aligned}
			Q_{\Phi}:L_p(\mathcal{M};(L_2(\mathbb{R},\bigg|\frac{dt}{t}\bigg|))^c)&\to L_p(\mathcal{M};(L_2(\mathbb{R},\bigg|\frac{dt}{t}\bigg|))^c)\\
			\{x_t\}_{t\in\mathbb{R}}&\mapsto \biggl\{\frac{f_\eta(t)(f_{\eta}(t)-\sqrt{A})^{-1}x_t}{2\pi\mathrm{i}}\biggr\}_{t\in\mathbb{R}}.
		\end{aligned}
	\end{equation*}
	From the proof of \cite[Theorem 4.14]{JX}, we know that 
	\begin{equation}\label{Qc}
		\|Q_c\|\le K_1K_2\|Q_{\Phi}\| \quad \text{and} \quad \|Q_{\Phi}\|\le \mathrm{Col}(\mathcal{R}),
	\end{equation}
	where $$K_1=\int_{\Gamma_\eta}|F(z)||\frac{dz}{z}|,\quad K_2=\int_{\Gamma_\eta}|G(z)||\frac{dz}{z}|$$ and 
	\begin{equation*}
		\mathcal{R}=\bigg\{\frac{1}{\mu(I)}\int_I f_\eta(t)(f_{\eta}(t)-\sqrt{A})^{-1} d\mu(t): I\subset\mathbb{R},0<\mu(I)<\infty\bigg\}.
	\end{equation*}
	Here $d\mu(t)=|\frac{dt}{t}|$. Notice that $\frac{\pi}{4}<\eta<\xi$. Using Lemma \ref{TFbound} and Proposition \ref{mainlemma}, one has
	\begin{equation}\label{QPhi}
		\|Q_{\Phi}\|\le \mathrm{Col}(\mathcal{R})\le \mathrm{Col}(\{z(z-\sqrt{A})^{-1}:z\in\Gamma_\eta\})\lesssim p'.
	\end{equation}
	Besides, by direct calculations, one has
	\begin{equation*}
		\begin{aligned}
			K_1=\int_{\Gamma_\eta}|F(z)|\bigg|\frac{dz}{z}\bigg|=2\int_0^\infty e^{-t\mathrm{cos}\eta}dt\lesssim \frac{1}{\mathrm{cos}\eta}\le \frac{1}{\mathrm{cos}(3\pi/8)}
		\end{aligned}
	\end{equation*}
	and $K_2=4K_1$. Hence from \eqref{Qc} and \eqref{QPhi} we have
	\begin{equation*}
		\|Q_c\|\lesssim p'.
	\end{equation*}
	Similarly,
	\begin{equation*}
		\|Q_r\|\lesssim p'.
	\end{equation*}
	Therefore, by \eqref{ppppp} and \eqref{xQcQr}, we obtain
	\begin{equation}\label{xppt}
		\|x\|_{p,P}\lesssim p'^{3}\|x\|_{p}
	\end{equation}
and 
\begin{equation*}
	\beta_p\lesssim p'^{3}.
\end{equation*}

    Finally we estimate $\alpha_p$ for $1<p<2$ by duality. Let $x\in L_p(\mathcal{M})$ and $y\in L_{p'}(\mathcal{M})$. 
    It has been proved in Section \ref{est1} that
    \begin{equation}\label{yy1y2y3}
    	\|y\|_{p',P}\le \beta_{p'}\|y\|_{p'}.
    \end{equation}
    Let $F(z)=-ze^{-z}$. From \eqref{xfA}, we see
    \begin{equation*}
    	x=4\int_0^\infty F(t\sqrt{A})F(t\sqrt{A})(x)\frac{dt}{t}.
    \end{equation*}
    Then by the duality in \eqref{duality1} and \eqref{yy1y2y3}, for any decomposition $x=x_1+x_2$ one has
    \begin{equation*}
    	\begin{aligned}
    		|\langle x,y\rangle|{}&=4\bigg|\int_0^\infty \langle F(t\sqrt{A})(x), F(t\sqrt{A})(y)\rangle \frac{dt}{t}\bigg|\\
    		&\le 4\bigg|\int_0^\infty \langle F(t\sqrt{A})(x_1), F(t\sqrt{A})(y)\rangle \frac{dt}{t}\bigg|+4\bigg|\int_0^\infty \langle F(t\sqrt{A})(x_2), F(t\sqrt{A})(y)\rangle \frac{dt}{t}\bigg|\\
    		&\le 4\big(\|G_c^{P}(x_1)\|_{p}\|G_r^{P}(y)\|_{p'}+\|G_r^{P}(x_2)\|_{p}\|G_c^{P}(y)\|_{p'}\big)\\
    		&\le 4\beta_{p'}\big(\|G_c^{P}(x_1)\|_{p}+\|G_r^{P}(x_2)\|_{p}\big)\|y\|_{p'}.
    	\end{aligned}
    \end{equation*}
    Hence we have
    \begin{equation*}
    	\|x\|_{p}\le 4\beta_{p'}\|x\|_{p,P},
    \end{equation*}
which implies that 
\begin{equation*}
	\alpha_p\lesssim p'.
\end{equation*}
The proof is completed.
\end{proof}

Finally we prove Corollary \ref{widealpha}.
\begin{proof}[Proof of Corollary \ref{widealpha}]
        For any two functions $u_1,u_2$ with $u_1(t)+u_2(t)=t\partial P_t(x)=F(t\sqrt{A})(x)$ for any $t\in\mathbb{R}_+$, by the duality in \eqref{duality1} and \eqref{yy1y2y3}, we estimate
     \begin{equation*}
     	\begin{aligned}
     		|\langle x,y\rangle|{}&=4\bigg|\int_0^\infty \langle F(t\sqrt{A})(x), F(t\sqrt{A})(y)\rangle \frac{dt}{t}\bigg|\\
     		&\le 4\bigg|\int_0^\infty \langle u_1(t), F(t\sqrt{A})(y)\rangle \frac{dt}{t}\bigg|+4\bigg|\int_0^\infty \langle u_2(t), F(t\sqrt{A})(y)\rangle \frac{dt}{t}\bigg|\\
     		&\le 4\big(\|u_1\|_{L_p(\mathcal{M};(L_2(\mathbb{R}_+,\frac{dt}{t}))^c)}\|G_r^{P}(y)\|_{p'}+\|u_2\|_{L_p(\mathcal{M};(L_2(\mathbb{R}_+,\frac{dt}{t}))^r)}\|G_c^{P}(y)\|_{p'}\big)\\
     		&\le 4\beta_{p'}\big(\|u_1\|_{L_p(\mathcal{M};(L_2(\mathbb{R}_+,\frac{dt}{t}))^c)}+\|u_2\|_{L_p(\mathcal{M};(L_2(\mathbb{R}_+,\frac{dt}{t}))^r)}\big)\|y\|_{p'}.
     	\end{aligned}
     \end{equation*}
     Therefore, 
     \begin{equation*}
     	\|x\|_{p}\le 4\beta_{p'}[x]_{p,P}.
     \end{equation*}
     We deduce that 
     \begin{equation*}
     	\widetilde{\alpha}_p\lesssim p'.
     \end{equation*}  
\end{proof}

{\textbf{Acknowledgments.}} We thank Professor Quanhua Xu for proposing to us the subject of this article.
This work is partially supported by the French ANR project (No. ANR-19-CE40-0002).

\bibliographystyle{myrefstyle}
\bibliography{semiref999}

\end{document}